\documentclass{amsart} \usepackage{amssymb,tikz-cd,tikz,enumitem}
\usepackage{multicol} \usetikzlibrary{positioning}

\numberwithin{equation}{subsection}

\title[Bott-Samelson moment graph and quantum cohomology] {The moment graph for
Bott-Samelson varieties and applications to quantum cohomology} \author{Camron
Withrow} \address{Department of Mathematics, Virginia Tech, 460 McBryde Hall,
Blacksburg, VA 24060} \email{cwithrow@vt.edu}

\newtheorem{mainthm}{Theorem} \newtheorem{maincor}{Corollary}[mainthm]

\newtheorem{prop}{Proposition}[section] \newtheorem{thm}{Theorem}[section]
\newtheorem{lemma}{Lemma}[section] \newtheorem{cor}{Corollary}[section]

\theoremstyle{definition} \newtheorem{defn}{Definition}[section]
\newtheorem{example}{Example}[section] \newtheorem{remark}{Remark}[section]
\newtheorem*{thank}{Acknowledgments}

\newcommand{\Z}{\mathbb{Z}} \newcommand{\Q}{\mathbb{Q}}
\newcommand{\R}{\mathbb{R}} \newcommand{\C}{\mathbb{C}}
\renewcommand{\P}{\mathbb{P}} \renewcommand{\epsilon}{\varepsilon}
\newcommand{\w}{\mathfrak{w}} \renewcommand{\t}{\mathfrak{t}}
\newcommand{\codim}{\text{codim}\>} \newcommand{\virt}{\text{virt}}
\newcommand{\vdim}{\text{vdim}\,}
\newcommand{\moduli}[3]{\overline{M}_{0,#1}(#2,#3)}
\newcommand{\GW}[2]{I_{#2}(#1)}

\begin{document}

\begin{abstract} We give a description of the moment graph for Bott-Samelson
		varieties in arbitrary Lie type. We use this, along with curve
		neighborhoods and explicit moduli space computations, to compute a
		presentation for the small quantum cohomology ring of a particular
Bott-Samelson variety in Type $A$. We also show the conjecture $\mathcal{O}$ of
Galkin, Golyshev, and Iritani holds for that Bott-Samelson variety.
\end{abstract} \maketitle \tableofcontents

\section{Introduction} \label{intro}
The (small) quantum cohomology of homogeneous varieties has been studied
extensively due to its connection with questions in enumerative geometry; see
the introduction in \cite{FultonPandharipande} for a discussion of enumerative
results relating to the quantum cohomology of the projective plane. The key to
obtaining a presentation for the quantum cohomology ring is to quantize the
relations in the ordinary cohomology ring. For flag manifolds $G/B$, this
quantization involves the Toda lattice; see \cite{Kim}.

More generally, the quantum cohomology of toric varieties is well understood
(originally due to Batyrev \cite{Batyrev}; see \cite{ToricQH} for a modern
discussion, Proposition 2.5 in particular.) Vakil in \cite{gwsurfaces} computed
Gromov-Witten invariants for Hirzebruch surfaces and (all but two) del Pezzo
surfaces in all genera, and shows these invariants are enumerative; see Section
8 in \cite{gwsurfaces} for a discussion of curve counts for Hirzebruch
surfaces. Little is known about the quantum cohomology of varieties which are
not convex and/or not toric. For example, see \cite{Pech} and \cite{Shifler}.
Bott-Samelson varieties are generally not convex in dimension greater than one,
and are generally not toric in dimension greater than two, however they are
intimately related to homogeneous spaces $G/B$, so it is natural consider their
quantum cohomology.

In this article, we describe the moment graph of Bott-Samelson varieties with a
view towards describing the (small) quantum cohomology ring of Bott-Samelson
varieties. If $X$ is a variety on which an algebraic torus $T$ acts with
finitely many fixed points, the moment graph is defined as follows: the
vertices are the set of fixed points $X^T$, and two vertices $x,y\in X^T$ are
connected by an edge if there is a $T$-stable curve containing both $x,y$. 

In order to describe our results more concretely, we recall some constructions
and fix some notation. Let $G$ be a simple Lie group over
$\C$. Fix a maximal torus contained in a Borel subgroup $T\subset B\subset G$;
the Weyl group is denoted $W:=N_G(T)/T$. The associated root system is denoted
$\Phi=\Phi(G,T)$, the base corresponding to the fixed Borel subgroup is denoted
$\Delta\subset \Phi$, and the set of positive roots is denoted $\Phi^+$. The
minimal parabolic subgroup corresponding to $\alpha\in \Delta$ is denoted
$P_{\alpha}$. 

For a sequence of simple reflections
$(s_{\alpha_1},s_{\alpha_2},\ldots,s_{\alpha_n})\in W^n$, the corresponding
Bott-Samelson variety is denoted $Z=Z(\alpha_1,\ldots,\alpha_n)$. A
Bott-Samelson variety is a tower of $\P^1$-bundles $$\begin{tikzcd}
Z(\alpha_1,\ldots,\alpha_n) \arrow{r}{\pi_n} & Z(\alpha_1,\ldots,\alpha_{n-1})
\arrow{r}{\pi_{n-1}} & \cdots \arrow{r}{\pi_2} & Z(\alpha_1) \arrow{r}{\pi_1} &
\{pt\} \end{tikzcd}$$ where each bundle has a natural section $s_k:
Z(\alpha_1,\ldots,\alpha_{k-1})\to Z(\alpha_1,\ldots,\alpha_k)$, and each
Bott-Samelson variety has a morphism $\theta_k: Z(\alpha_1,\ldots,\alpha_k)\to
G/B$.


Bott-Samelson varieties are $T$-varieties (this is described in Section
\ref{prelims}.) The fixed point set $Z^T$ is easy to describe; the $T$-fixed
points correspond to subsequences of $(s_{\alpha_1},\ldots,s_{\alpha_n})$. The
combinatorial object which corresponds to the $T$-fixed points are
$\varepsilon\in \{0,1\}^n$; for $x\in Z^T$, we will denote the binary $n$-tuple
corresponding to $x$ by $\varepsilon_x$. 

The $n$-tuple $\varepsilon_x$ can be described inductively as follows:
$\varepsilon_x$ is obtained from the $(n-1)$-tuple $\varepsilon_{\pi(x)}$ by
appending either a zero or one according as $x\in s(Z')$ or $x\notin s(Z')$
respectively. The next definition comes from \cite[Section 2]{Willems04}; we
have slightly modified the notation.

\begin{defn} \label{weyl prod} For $\varepsilon\in \{0,1\}^n$, denote by
	$\pi_+(\varepsilon)$ the set of entries $i$ such that $\varepsilon_i=1$.
	Define $$w_k(\varepsilon) = \prod_{\substack{1\leq i\leq k\\i\in
	\pi_+(\varepsilon)}}s_{\alpha_i}$$ ($w_k(\varepsilon) = 1$ if $\{1\leq i \leq
k, i\in \pi_+(\varepsilon)\}=\varnothing$), set
$w(\varepsilon)=w_n(\varepsilon)$, and define $\varepsilon(\alpha_k) =
w_k(\varepsilon)\alpha_k\in \Phi$ for each $1\leq k\leq n$.  \end{defn}

The next definition gives us a notation for the fixed points which lie on the
same fiber as a given fixed point.

\begin{defn} For $x\in Z^T$, define $\varepsilon_x^0$ and
	$\varepsilon_x^{\infty}$ by adjoining either a 0 or 1 respectively to
	$\varepsilon_{\pi(x)}$. Note, $\varepsilon_x^0$ corresponds to the $T$-fixed
	point in $s(Z')$ which is also contained in the fiber of $\pi:Z\to Z'$
containing $x$, and $\varepsilon_x^{\infty}$ corresponds to the other $T$-fixed
point in that fiber.  \end{defn}

Our main theorem is an inductive characterization of the set of $T$-stable
curves in $Z$. Since any $T$-stable curve is a union of irreducible $T$-stable
curves, we characterize the points $x,y\in Z^T$ which are joined by irreducible
$T$-stable curves.

\begin{mainthm}\label{main theorem} Let $x,y\in Z^T$, and let $k$ be the first
	index where $\varepsilon_x,\varepsilon_y$ differ. If $k=n$, then $x,y$
	are joined by an $T$-stable fiber of $\pi:Z\to Z'$. Otherwise, $k<n$
	and we suppose that $\pi(x),\pi(y)$ are joined by an irreducible
	$T$-stable curve. Let $C$ be one such $T$-stable curve joining
	$\pi(x),\pi(y)$, and let $h$ be the class of the fiber of $\pi$.

There are four possibilities for the restriction of the moment graph to
	$\{\varepsilon_x^0,\varepsilon_x^{\infty},\varepsilon_y^0,\varepsilon_y^{\infty}\}$:

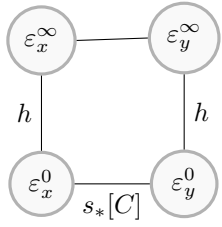
\begin{multicols}{2} Case I.  \begin{center}
	\begin{tikzpicture}[roundnode/.style={circle, draw=gray!60, fill=gray!5, very
		thick, minimum size=7mm}, ]
\node[roundnode]	(basepoint) {$\varepsilon_x^0$}; \node[roundnode] (upper)
		[above=of basepoint]		{$\varepsilon_x^{\infty}$}; \node[roundnode]
		(right)		[right=of basepoint] {$\varepsilon_y^0$}; \node[roundnode]
		(upright) [above=of right] {$\varepsilon_y^{\infty}$};

\draw (upper.south) -- (basepoint.north) 	node[pos=0.5,left] {$h$}; \draw
	(upper.east) -- (upright.west); \draw (basepoint.east) -- (right.west)
	node[pos=0.5,below] 	{$s_*[C]$}; \draw (upright.south) -- (right.north)
	node[pos=0.5,right] {$h$}; \end{tikzpicture} \end{center}

Case II.  
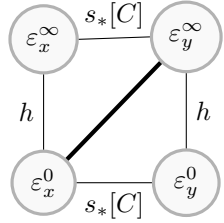
\begin{center} \begin{tikzpicture}[roundnode/.style={circle,
	draw=gray!60, fill=gray!5, very thick, minimum size=7mm}, ]
\node[roundnode]	(basepoint) {$\varepsilon_x^0$}; \node[roundnode] (upper)
	[above=of basepoint]		{$\varepsilon_x^{\infty}$}; \node[roundnode] (right)
	[right=of basepoint] {$\varepsilon_y^0$}; \node[roundnode]	(upright)
	[above=of right] {$\varepsilon_y^{\infty}$};

\draw (upper.south) -- (basepoint.north) node[pos=0.5,left]	{$h$}; \draw[ultra
	thick] (basepoint.east) -- (right.west) node[pos=0.5,below]	{$s_*[C]$}; \draw
	(upright.south) -- (right.north) node[pos=0.5,right]	{$h$}; \draw
	(upper.south east) -- (right.north west); \draw (basepoint.north east) --
(upright.south west); \end{tikzpicture} \end{center}

Case III.  
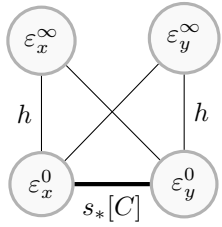
\begin{center} \begin{tikzpicture}[roundnode/.style={circle,
	draw=gray!60, fill=gray!5, very thick, minimum size=7mm}, ]
\node[roundnode]	(basepoint) {$\varepsilon_x^0$}; \node[roundnode] (upper)
	[above=of basepoint]		{$\varepsilon_x^{\infty}$}; \node[roundnode] (right)
	[right=of basepoint] {$\varepsilon_y^0$}; \node[roundnode]	(upright)
	[above=of right] {$\varepsilon_y^{\infty}$};

\draw (upper.south) -- (basepoint.north) node[pos=0.5,left]		{$h$}; \draw
	(upper.east) -- (upright.west) node[pos=0.5,above] {$s_*[C]$}; \draw
	(basepoint.east) -- (right.west) node[pos=0.5,below]		{$s_*[C]$}; \draw
	(upright.south) -- (right.north) node[pos=0.5,right]		{$h$}; \draw[ultra
	thick] (basepoint.north east) -- (upright.south west); \end{tikzpicture}
\end{center}

Case IV.  
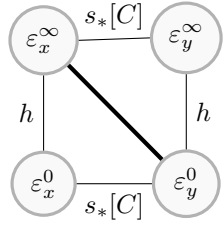
\begin{center} \begin{tikzpicture}[roundnode/.style={circle,
	draw=gray!60, fill=gray!5, very thick, minimum size=7mm}, ]
\node[roundnode]	(basepoint) {$\varepsilon_x^0$}; \node[roundnode] (upper)
	[above=of basepoint]		{$\varepsilon_x^{\infty}$}; \node[roundnode] (right)
	[right=of basepoint] {$\varepsilon_y^0$}; \node[roundnode]	(upright)
	[above=of right] {$\varepsilon_y^{\infty}$};

\draw (upper.south) -- (basepoint.north) node[pos=0.5,left]		{$h$}; \draw
	(upper.east) -- (upright.west) node[pos=0.5,above] {$s_*[C]$}; \draw
	(basepoint.east) -- (right.west) node[pos=0.5,below]		{$s_*[C]$}; \draw
	(upright.south) -- (right.north) node[pos=0.5,right]		{$h$}; \draw[ultra
	thick] (upper.south east) -- (right.north west); \end{tikzpicture}
	\end{center} \end{multicols}

The cases are characterized as follows: \begin{enumerate}[itemsep=1ex]
	\item[I.] $\varepsilon_x^0(\alpha_k)\neq \varepsilon_x^0(\alpha_n)$ and
		$\varepsilon_y^0(\alpha_k)\neq \varepsilon_y^0(\alpha_n)$;

\item[II.] $\varepsilon_x^0(\alpha_k) = \varepsilon_x^0(\alpha_n)$ and
	$\varepsilon_y^0(\alpha_k) = \varepsilon_y^0(\alpha_n)$;

\item[III.] $\varepsilon_x^0(\alpha_k) = \varepsilon_x^0(\alpha_n)$ and
	$\varepsilon_y^0(\alpha_k)\neq \varepsilon_y^0(\alpha_n)$;

\item[IV.] $\varepsilon_x^0(\alpha_k)\neq \varepsilon_x^0(\alpha_n)$ and
$\varepsilon_y^0(\alpha_k) = \varepsilon_y^0(\alpha_n)$; \end{enumerate}

In each case, the unlabeled curves have the same homology class, which are as
follows: \begin{enumerate}[itemsep=1ex] \item[I.] \begin{align*}\begin{cases}
			s_*[C] - (\alpha_k,\alpha_n^{\vee})h, &w(\varepsilon_x^0)\neq
			w(\varepsilon_y^0)\\ s_*[C], &w(\varepsilon_x^0) = w(\varepsilon_y^0)
\end{cases}\end{align*}

\item[II.] $$s_*[C] - h$$

\item[III.] $$s_*[C] + h$$

\item[IV.] $$s_*[C] + h$$ \end{enumerate}

In cases II, III, and IV, the bold line indicates there is a one-dimensional
family of $T$-stable curves joining those fixed points.  \end{mainthm}

\begin{remark} Since $\pi:Z\to Z'$ is proper, if $x,y$ are joined by a
	$T$-stable curve, then $\pi(x),\pi(y)$ are either joined by a $T$-stable
	curve, or $\pi(x)=\pi(y)$. In particular, the inductive hypothesis in Theorem
	\ref{main theorem} is a necessary condition for $x,y$ to be joined by a
	$T$-stable curve.

Moreover, since $\pi:Z\to Z'$ is a $\P^1$-bundle with a section $s$, the
push-forward $s_*:H_*(Z')\to H_*(Z)$ is the inclusion $H_*(Z')\subset H_*(Z)$.
Thus, the homology classes of $T$-stable curves are characterized inductively
by our theorem.  \end{remark}

\begin{example} In type $A_2$, the Bott-Samelson variety $Z(\alpha_1,\alpha_2)$
	is the Hirzebruch surface $\mathbb{F}_1$, and hence the moment graph is
	
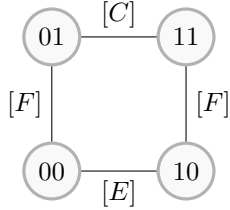
\begin{center} \begin{tikzpicture}[roundnode/.style={circle, draw=gray!60,
		fill=gray!5, very thick, minimum size=7mm}, ]
\node[roundnode]	(basepoint) {00}; \node[roundnode]	(upper)		[above=of
		basepoint] {01}; \node[roundnode] 	(right)		[right=of basepoint] {10};
		\node[roundnode]	(upright)		[above=of right] {11};

\draw (upper.south) -- (basepoint.north) 	node[pos=0.5,left] {$[F]$}; \draw
		(upper.east) -- (upright.west) node[pos=0.5,above]	{$[C]$}; \draw
		(basepoint.east) -- (right.west)		node[pos=0.5,below] {$[E]$}; \draw
		(upright.south) -- (right.north) node[pos=0.5,right]		{$[F]$};
	\end{tikzpicture} \end{center} where $F$ is the fiber, $E$ is the exceptional
	divisor (i.e. the curve with self-intersection $-1$,) and $C$ is related to
	$E$ and $F$ in Pic $\mathbb{F}_1$ by $C = E+F$.

In the Bott-Samelson variety $Z(\alpha_1,\alpha_2,\alpha_1)$, the restriction
	of the moment graph to the fixed points $\{000,100,001,101\}$ is:
	
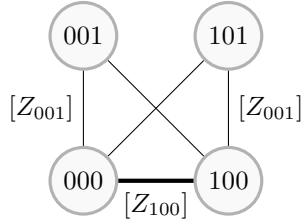
\begin{center} \begin{tikzpicture}[roundnode/.style={circle, draw=gray!60,
		fill=gray!5, very thick, minimum size=7mm}, ]
\node[roundnode]	(basepoint) {$000$}; \node[roundnode]	(upper)		[above=of
		basepoint] {$001$}; \node[roundnode] 	(right)		[right=of basepoint]
		{$100$}; \node[roundnode]	(upright)		[above=of right]			{$101$};

\draw (upper.south) -- (basepoint.north) node[pos=0.5,left]	{$[Z_{001}]$};
		\draw[ultra thick] (basepoint.east) -- (right.west) node[pos=0.5,below]
		{$[Z_{100}]$}; \draw (upright.south) -- (right.north) node[pos=0.5,right]
		{$[Z_{001}]$}; \draw (upper.south east) -- (right.north west); \draw
	(basepoint.north east) -- (upright.south west); \end{tikzpicture}
\end{center} where the diagonal curves have homology class $[Z_{100}] -
[Z_{001}]$. We draw the complete moment graph for this Bott-Samelson variety in
Example \ref{momentgraph picture}.

In the Bott-Samelson variety $Z(\alpha_1,\alpha_2,\alpha_1,\alpha_2)$, the
moment graph restricted to the fixed points $\{0000,1010,0001,1011\}$ is:
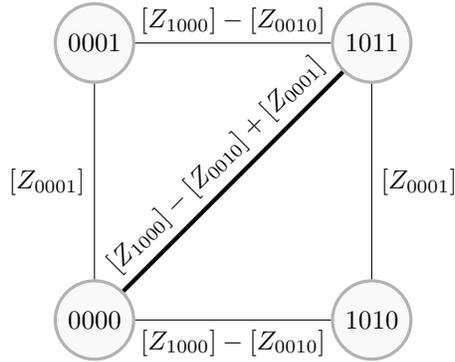
\begin{center} \begin{tikzpicture}[node distance=2.6cm,
	roundnode/.style={circle, draw=gray!60, fill=gray!5, very thick, minimum
	size=7mm}, ]
\node[roundnode]	(basepoint) {$0000$}; \node[roundnode]	(upper)		[above=of
	basepoint] {$0001$}; \node[roundnode] 	(right)		[right=of basepoint]
	{$1010$}; \node[roundnode]	(upright)		[above=of right]			{$1011$};

\draw (upper.south) -- (basepoint.north) node[pos=0.5,left]		{$[Z_{0001}]$};
	\draw (upper.east) -- (upright.west) node[pos=0.5,above]
	{$[Z_{1000}]-[Z_{0010}]$}; \draw (basepoint.east) -- (right.west)
	node[pos=0.5,below]		{$[Z_{1000}]-[Z_{0010}]$}; \draw (upright.south) --
	(right.north) node[pos=0.5,right]		{$[Z_{0001}]$}; \draw[ultra thick]
	(basepoint.north east) -- (upright.south west) node[pos=0.5,above,sloped]
	{$[Z_{1000}]-[Z_{0010}]+[Z_{0001}]$}; \end{tikzpicture} \end{center}

All homology classes have been expressed in the Bott-Samelson subvariety basis,
which we describe in Section \ref{prelims}.  \end{example}

In order to obtain a presentation for the quantum cohomology of Bott-Samelson
varieties, we need to show that certain Gromov-Witten invariants vanish, and we
need to compute some nonzero Gromov-Witten invariants. We address the non-zero
invariant calculations first. 

The Gromov-Witten invariants $I_{\beta}(\gamma_1,\ldots,\gamma_n)$ are defined
by intersection theory in the moduli space. However, in general we cannot
control the geometry of the moduli space $\overline{M}_{0,1}(Z,\beta)$. Under
some conditions on the curve class $\beta$ and the Lie group $G$, we are
able to prove the moduli space $\overline{M}_{0,1}(Z,\beta)$ is smooth, and so
we are able to perform the intersection theory calculations directly to
determine certain Gromov-Witten invariants.

\begin{mainthm} \label{modulispace thm} Suppose $\beta$ is indecomposable and
	effective, and suppose $G$ is of simply laced type. Then
	$\overline{M}_{0,1}(Z,\beta)$ is unobstructed; that is,
	$\overline{M}_{0,1}(Z,\beta)$ is smooth, irreducible, and has the expected
dimension $$\dim \moduli{1}{Z}{\beta} = \dim Z + \int_{\beta}c_1(T_Z) - 2.$$
\end{mainthm}

Theorem \ref{modulispace thm} has an important corollary which allows us to
carry out the necessary calculations.

\begin{maincor} \label{fiber cor} If $h\in H_2(Z)$ is the class of the fiber of
$\pi:Z\to Z'$, then $ev:\overline{M}_{0,1}(Z,h)\to Z$ is an isomorphism.
\end{maincor}

Corollary \ref{fiber cor} lets us convert intersection theory calculations in
the moduli space, into intersection theory calculations on the Bott-Samelson
variety $Z$.

With the explicit calculations described above, and the vanishing of certain
Gromov-Witten invariants that will be discussed in a few paragraphs, the final
ingredient in obtaining a presentation for the (small) quantum cohomology of
$Z(\alpha_1,\alpha_2,\alpha_1)$ is a brute-force calculation. Curve
neighborhood techniques and the moduli space results allow us to compute some
of the necessary Gromov-Witten invariants to quantize the relations in the
ordinary cohomology. The remaining unknown invariants (of which there are 111),
save for one, can be computed simply by imposing the relations that the quantum
cohomology ring is commutative (this is a system of 192 (generally) nonlinear
equations in 111 unknowns). The final Gromov-Witten invariant is computed using
a technique of Manolache \cite[Section 5.4]{Manolache}.

\begin{mainthm} \label{quantum relations} Let
	$Z=Z(\alpha_1,\alpha_2,\alpha_1)$. The (small) quantum cohomology ring
	$QH^*(Z)$ is isomorphic to a quotient of 
	$\Z[\sigma_{100},\sigma_{010},\sigma_{001},q_1,q_2,q_3]$, subject to the
	following relations: \begin{align*} \sigma_{100}^2 &= q_1q_3 -
		q_3\sigma_{100} + q_3\sigma_{010}\\ \sigma_{010}^2 &= q_1q_3 +
		2q_1\sigma_{100} - q_1\sigma_{010} + q_1\sigma_{001} + \sigma_{110}\\
		\sigma_{001}^2 &= q_1q_3 + q_2 - q_3\sigma_{100} + q_3\sigma_{010} -
	2\sigma_{101} + \sigma_{011} \end{align*} \end{mainthm}

Under this isomorphism, the generators $\sigma_{100},\sigma_{010},\sigma_{001}$
	are Poincar\'{e} dual to certain Bott-Samelson subvarieties in $Z$. For
	example, $\sigma_{001}$ is dual to the fiber of $\pi:Z\to Z'$. The other
	two classes $\sigma_{100}$, $\sigma_{010}$ arise in a similar way. The
	quantum parameters $q_1,q_2,q_3$ correspond curve classes
	$\beta_1,\beta_2,\beta_3$ which generate the cone of effective
	1-cycles.



The connection between Theorem \ref{main theorem} and quantum cohomology is
	given by curve neighborhoods, which will allow us to show that certain
	Gromov-Witten invariants vanish; curve neighborhoods were used to study the
	quantum cohomology and quantum $K$-theory of homogeneous spaces and ``almost
	homogeneous'' spaces (see \cite{curvenbhds}, \cite{BCMP16}, \cite{Shifler},
	\cite{Mare}). 
	
	Given an effective curve class $\beta$, and a closed subvariety
	$\Omega\subset Z$, the curve neighborhood $\Gamma_{\beta}(\Omega)$ is
	the union of all curves of class $\beta$ which intersect $\Omega$; our
	definition later will be written in terms of the moduli space of stable
	maps, but is equivalent to the one just given. It will be clear from
	the alternative definition that curve neighborhoods are closed.
	Moreover, it is clear if $\Omega$ is $B$-stable, then
	$\Gamma_{\beta}(\Omega)$ is also $B$-stable.  Finding all $B$-stable
	subvarieties in a Bott-Samelson variety is more difficult than finding
	all $B$-stable subvarieties in a homogeneous space. 
However, as we will show in Section \ref{curve nbhds}, the unexpected
	$B$-stable curves collapse under $\theta:Z\to G/B$. So, with some {\it
	ad hoc} arguments, we can use the moment graph to describe curve
	neighborhoods for Bott-Samelson varieties.

\begin{thank} I would like to thank my advisor, Leonardo Mihalcea, for his help
and encouragement throughout this project.  \end{thank}

\section{Preliminaries} \label{prelims}
Our main reference in this section is \cite[Chapter 2]{BrionKumar}. Given a
sequence of simple reflections $\w =
(s_{\beta_1},s_{\beta_2},\ldots,s_{\beta_k})$, consider the space
$P_{\w}:=P_{\beta_1}\times\cdots\times P_{\beta_k}$ equipped with the
$B^k$-action $$(b_1,\ldots,b_k)\odot (p_1,\ldots,p_r) =
(p_1b_1^{-1},b_1^{-1}p_2b_2,\ldots,b_{k-1}^{-1}p_kb_k)$$

\begin{defn} The {\it Bott-Samelson variety} $Z_{\w}$ is the coset space
	$$Z_{\w}:=P_{\w}/B^k$$ The points in $Z_{\w}$ will be denoted by
	$[p_1,\ldots,p_k]$. 

There is a natural $B$-action given by $$b.[p_1,p_2,\ldots,p_k] =
	[bp_1,p_2,\ldots,p_k]$$

$Z_{\w}$ contains an affine open cell $Z_{\w}^{\circ}$ defined by
	$$Z_{\w}^{\circ} := Bs_{\beta_1}B\times\cdots\times Bs_{\beta_k}B/B^k$$

As in the introduction, we index the subwords of $\w$ by binary $k$-tuples
	$\varepsilon\in \{0,1\}^k$.  For example, if $\w=(s_1,s_2,s_1)$ and
	$\varepsilon=(1,1,0)$, then $\w(\varepsilon) = (s_1,s_2)$. For the same
	$\w$, if $\varepsilon = (1,0,1)$, then $\w(\varepsilon) = (s_1,s_1)$. 

For any subword $\w(\varepsilon)$, there is a natural morphism
	$\pi_{\varepsilon}: Z_{\w}\to Z_{\w(\varepsilon)}$; if $\w(\varepsilon)$ is
	the initial subword of length $m$, we will denote $Z_{\w(\varepsilon)}$ by
	$Z_{\w[m]}$ and $\pi_{\varepsilon}$ by $\pi_{m}$.  

The {\it length} of $\varepsilon$, denoted $\ell(\varepsilon)$, is the number
	of ones in $\varepsilon$. If $\ell(\varepsilon)=1$, we will denote
	$\varepsilon=(i)$ where $i$ is the nonzero entry of $\varepsilon$. When
	$\varepsilon,\varepsilon'$ have no common components, we say they are {\it
	transverse}, denoted $\varepsilon\perp\varepsilon'$.

For each word $\w$, the product of the simple reflections (in order) is an
element of the Weyl group which we denote by $w(\w)$; see Definition \ref{weyl
prod}.  \end{defn}

Proofs for all the statements in the following proposition (with one small
exception) can be found in \cite[pp. 64-67]{BrionKumar}.  \begin{prop} Let
	$Z_{\w}$ be a Bott-Samelson variety, $X=G/B$ the flag variety, and let
	$\w(\varepsilon)$ be a subword of $\w$.  \begin{enumerate}[label=(\alph*)]
		\item $Z_{\w}$ is a smooth, projective variety.

\item The natural morphism $\pi:Z_{\w}\to Z_{\w[k-1]}$ defined by
	$$\pi([p_1,\ldots,p_k]) = [p_1,\ldots,p_{k-1}]$$ is $B$-equivariant, and
			realizes $Z_{\w}$ as a $\P^1$-bundle over $Z_{\w[k-1]}$.

\item The map $\theta_{\w}:Z_{\w}\to X$ defined by
	$$\theta_{\w}([p_1,\ldots,p_k]) = (p_1p_2\cdots p_k)B$$ is a $B$-equivariant
			morphism. Moreover, if $s_{\beta_1}s_{\beta_2}\cdots s_{\beta_k} = w(\w)$
			is a reduced word decomposition, then $\theta_{\w}$ is a birational
			equivalence: $Z_{\w}^{\circ}\xrightarrow{\simeq} X(w(\w))^{\circ}$.

\item The map $j_{\varepsilon}:Z_{\w(\varepsilon)}\to Z_{\w}$ defined by
	$$j_{\varepsilon}([p_1,\ldots,p_{\ell}]) =
			[1,1,\ldots,p_1,1,\ldots,p_{\ell},1,\ldots,1]$$ (where ones are placed in
			the components where $\varepsilon$ is zero) is a $B$-equivariant closed
			immersion. 

For the $(k-1)$-initial subword, the morphism will be denoted by
			$s_{\w}:Z_{\w[k-1]}\to Z_{\w}$ and is a section of $\pi_{\w}$.

\item The natural commutative diagram $$\begin{tikzcd} Z_{\w}
\arrow{d}[swap]{\pi_{\w}} \arrow{rr}{\theta_{\w}}	& & X
\arrow{d}{p_{\beta_k}}\\ Z_{\w[k-1]} \arrow{rr}{p_{\beta_k}\theta_{\w[k-1]}} &
& G/P_{\beta_k} \end{tikzcd}$$ is Cartesian; that is, $Z_{\w}$ is the fiber
			product $Z_{\w[k-1]}\times_{G/P_{\beta_k}} X$. The $B$-action on the
	fiber product is diagonal.  \end{enumerate} \end{prop}

\begin{proof} As mentioned before the statement of the proposition, parts
	(a)-(d) are discussed in \cite[pp. 64-67]{BrionKumar}. Part (e) is Exercise
	2.2.E.1 in \cite{BrionKumar}, with the exception of the $B$-action statement.
	This follows easily since each of the maps $\pi_{\w}$ and $\theta_{\w}$ are
$B$-equivariant.  \end{proof}

The cells
$Z_{\varepsilon}^{\circ}:=j_{\varepsilon}(Z(\w(\varepsilon))^{\circ})$ form an
affine cell decomposition of $Z_{\w}$: $$Z_{\w} =
\bigcup_{\varepsilon}Z_{\varepsilon}^{\circ}$$ In particular,
$\{[Z_{\varepsilon}]:\varepsilon\in \{0,1\}^k\}$ is an additive basis for
$H^*(Z_{\w})$; the dual basis (under the Poincare pairing) is denoted
$\{\sigma_{\varepsilon}:\varepsilon\in\{0,1\}^k\}$.

A presentation for the (ordinary) cohomology of Bott-Samelson varieties was
obtained in \cite[Lemma 4.5]{Duan}; we record the result here.  \begin{prop}
	The cohomology of $Z_{\w}$ is generated by
	$\{\sigma_{\varepsilon}:\varepsilon\in\{0,1\}^k\}$ with relations
	\begin{align*} \sigma_{\varepsilon}\sigma_{\varepsilon'} &=
		\sigma_{\varepsilon+\varepsilon'},\quad\quad\text{if }
		\varepsilon\perp\varepsilon'\\ \sigma_{(j)}^2 &= \sum_{i<j}
	-(\alpha_i,\alpha_j^{\vee})\sigma_{(i)+(j)} \end{align*} where
$(\alpha,\beta^{\vee})$ is the usual root/coroot pairing.  \end{prop}

\begin{remark} It is known, though the author was unable to find a reference,
	that the dual classes $\sigma_{\varepsilon}$ are preserved under pullback
	along the morphisms $\pi_{\w[r]}: Z_{\w}\to Z_{\w[r]}$ ($1\leq r\leq k-1$).
	We will write formulas such as: $\pi_{\w[r]}^*(\sigma_{\varepsilon}) =
	\sigma_{\varepsilon}$. One should interpret $\varepsilon\in \{0,1\}^r$ as a
	binary $k$-tuple ($k>r$) by appending zeros to the end of $\varepsilon$. 

We include a proof that the dual classes are preserved under pull-back for
	completeness.

\begin{proof} It suffices to show that $\sigma_{(i)}$ is preserved under
	pullback for all $i$. Let $i<k$ and consider $\sigma_{(i)}\in H^*(Z')$.  Then
	$$\int_Z \pi^*\sigma_{(i)}\cdot [Z_{\varepsilon}] = \int_{Z'} \sigma_{(i)}
	\cdot \pi_*[Z_{\varepsilon}].$$ However, $Z_{\varepsilon}$ is either the
	preimage under $\pi$ of a Bott-Samelson subvariety in $Z'$, in which case
	$\pi_*[Z_{\varepsilon}] = 0$, or $Z_{\varepsilon}$ is the image of
	$Z'_{\varepsilon}$ under the canonical section $s:Z'\to Z$ and
$\pi_*[Z_{\varepsilon}] = [Z'_{\varepsilon}]$. Therefore, $\pi^*\sigma_{(i)}$
is Poincar\'{e} dual to $Z_{(i)}$, as claimed.  \end{proof} \end{remark}

To conclude this section, we define the cone of effective curves. 

\begin{defn} The {\it cone of effective curves} in $H_2(Z)$ is the set of all
effective $1$-cycles; that is, a positive combination of the fundamental
classes of irreducible curves in $Z$.  \end{defn}

In \cite[Lemma 2.1]{AndersonBottSamelson}, it is stated that for complete,
	irreducible varieties $X$, the cone of effective $k$-cycles on $X$ is
	generated by the classes of $B$-invariant $k$-cycles. Thus, the cone of
	effective curves for a Bott-Samelson variety is generated by the classes of
	the $B$-stable curves. We characterize those curves in the next
	section. First,	we give a quick proof of \cite[Lemma 2.1]{AndersonBottSamelson}.

\begin{proof}
	Let $d\in N_k(X)_{\R}$ be an irreducible, effective $k$-cycle on $X$,
	and let $C_d^X$ denote the Chow variety for $X$ of degree $d$. The
	$B$-action on $X$ naturally lifts to $C_d^X$, and by the Borel fixed
	point theorem, there is a $B$-fixed point in $C_d^X$. This fixed point
	corresponds to a $B$-invariant $k$-cycle on $X$ with degree $d$.

	Therefore, every $k$-cycle on $X$ can be written as a sum of
	$B$-invariant $k$-cycles.
\end{proof}

\section{The moment graph} \label{moment graph}
Bott-Samelson varieties are sympletic varieties with respect to the given torus
	action, and thus are equipped with a {\it moment map} $Z\to \t^*$; see
	\cite[Section 4.1]{Escobar} for more details on the moment map for
	Bott-Samelson varieties, along with a description of the images of the
	$T$-fixed points under this map. 

The image of the 1-skeleton of $Z$ under this map (that is, the image of the
	$T$-fixed points and $T$-stable curves) is called the {\it moment graph} for
	$Z$. The $T$-fixed points in a Bott-Samelson variety were discussed in
	Section \ref{prelims}, so it remains to describe the $T$-stable curves.

To begin, we characterize the $T$-stable curves on a $\P^1$-bundle over $\P^1$.
	Let $X(T)$ the character group of $T$, and suppose $p:\Sigma\to \P^1$ is a
	$T$-equivariant $\P^1$-bundle. Moreover, we assume the $T$-actions on $\P^1$
	and $\Sigma$ are nontrivial. For $x\in \Sigma^T$, the {\it weights} of
	$\Sigma$ at $x$ are $\chi,\psi\in X(T)$ where \begin{align*} t.v &=
	\chi(t)v,\quad v\in T_{p(x)}\P^1\\ t.w &= \psi(t)w,\quad w\in T_xF
	\end{align*} where $F$ is the (geometric) fiber $p^{-1}(x)$. 

\begin{lemma} \label{weights} There are infinitely many $T$-stable curves
passing through $x\in \Sigma^T$ if and only if the weights at $x$, $\chi$ and
$\psi$, are equal. Otherwise, there are exactly two (irreducible) $T$-stable
curves passing through $x$.  \end{lemma} \begin{proof} Since $\Sigma$ is smooth
	and projective, there is a $T$-stable affine open neighborhood of $x$ which
	is $T$-isomorphic to $T_x\Sigma$ (\cite[Theorem 2.5]{BB}.) Choose local
	coordinates $X,Y$ so that $\C[T_x\Sigma]\simeq \C[X,Y]$ with $t.X=\chi(t)X$
	and $t.Y=\psi(t)Y$. 

If $\chi=\psi$, then it is clear that the lines $V(X-\alpha Y)\subset
	T_x\Sigma$ ($\alpha\in \C$) are $T$-stable curves; in particular, there are
	infinitely many $T$-stable curves in $\Sigma$ passing through $x$.

If $\chi\neq \psi$, the characters are linearly independent. Therefore, the
only $T$-stable curves passing through $x$ are (in local coordinates)
$V(X),V(Y)$.  \end{proof}


If there is a $T$-fixed point $x\in \Sigma$ with repeated weights, it is also
easy to show there are exactly two $T$-fixed points with repeated weights.
Moreover, there are additional $T$-stable curves in $\Sigma$.

\begin{lemma} \label{sections} Let $\Sigma = \P^1\times \P^1$ (where each
	$\P^1$ is equipped with a nontrivial $T$-action). If the $T$-fixed points in
	$\Sigma$ all have distinct weights, then the $T$-stable fibers of the two
	projections are the only $T$-stable curves.  Otherwise, the images of the
sections $s_t: \P^1\to \Sigma$ ($t\in T$), defined by $s_t(z)=(z,t.z)$, are
also $T$-stable, and these exhaust the set of $T$-stable curves in $\Sigma$.
\end{lemma} \begin{proof} By the previous lemma, the only case that requires
	analysis is when there are repeated weights in $T_x\Sigma$ for some $x\in
	\Sigma^T$; fix such a point $x\in \Sigma$. Since $T_x\Sigma$ has equal
	weights, the $T$-actions on each factor $\Sigma=\P^1\times \P^1$ are equal.
	Thus, \begin{align*} s_t(t'.z) = (t'.z,t.(t'.z)) &= (t'.z,t'.(t.z))\\ &=
	t'.(z,t.z) \end{align*} that is, $s_t$ is a $T$-equivariant section of
	$p_1:\Sigma\to \P^1$.  

These curves exhaust the set of (irreducible) $T$-stable curves in $\Sigma$
	since any other $T$-stable curve $C$, aside from the $T$-stable fibers,
	intersects one of the sections $s_t$ in a point not fixed by $T$. Therefore,
	$C$ shares a dense open orbit with the section $s_t$ and so is equal to the
image $s_t(\P^1)$.  \end{proof}

\begin{example} In \cite[Section 2.1]{Magyar}, a configuration variety
	interpretation of Bott-Samelson varieties is provided in type $A$.  For the
	Bott-Samelson variety $Z(\alpha_1,\alpha_2,\alpha_1)$, the points correspond
	to configuration diagrams 
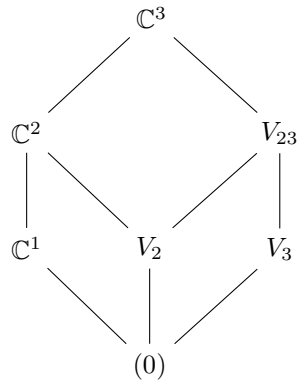
\begin{center} \begin{tikzpicture}
\node (top) at (0,1) {$\C^3$}; \node[below left = of top] (onetwo) {$\C^2$};
		\node[below right = of top] (twothree) {$V_{23}$}; \node[below = of onetwo]
		(one) {$\C^1$}; \node[below = of twothree] (three) {$V_3$}; \node[below
		right = of onetwo] (two) {$V_2$}; \node[below = of two] (zero) {$(0)$};
\draw (top) -- (onetwo); \draw (top) -- (twothree); \draw (onetwo) -- (two);
	\draw (twothree) -- (two); \draw (onetwo) -- (one); \draw (twothree) --
	(three); \draw (one) -- (zero); \draw (two) -- (zero); \draw (three) --
	(zero); \end{tikzpicture} \end{center} where $V_2,V_3, V_{23}$ are vector
	subspaces of $\C^3$, $\dim V_2 = \dim V_3 = 1$ and $\dim V_{23} = 2$, with a
	line between two subspaces meaning inclusion (so, $V_2\subset V_{23}$ and
	$V_2\subset \C^2$ in the standard basis). 

The sub Bott-Samelson variety $Z_{101}$  (note this subword is not reduced) is
	characterized by $V_{23}=\C^2$, so the configuration diagram is
	
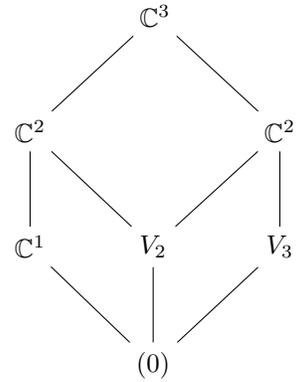
\begin{center} \begin{tikzpicture}
\node (top) at (0,1) {$\C^3$}; \node[below left = of top] (onetwo) {$\C^2$};
		\node[below right = of top] (twothree) {$\C^2$}; \node[below = of onetwo]
		(one) {$\C^1$}; \node[below = of twothree] (three) {$V_3$}; \node[below
		right = of onetwo] (two) {$V_2$}; \node[below = of two] (zero) {$(0)$};
\draw (top) -- (onetwo); \draw (top) -- (twothree); \draw (onetwo) -- (two);
	\draw (twothree) -- (two); \draw (onetwo) -- (one); \draw (twothree) --
	(three); \draw (one) -- (zero); \draw (two) -- (zero); \draw (three) --
	(zero); \end{tikzpicture} \end{center} Thus $Z_{101}\simeq \P^1\times \P^1$;
	this isomorphism is the natural morphism $$
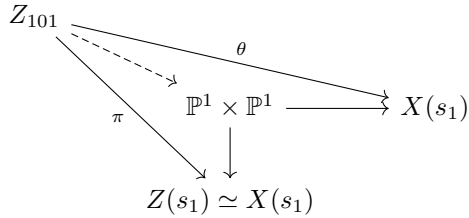
\begin{tikzcd} Z_{101}
	\arrow{ddr}[swap]{\pi} \arrow{drr}{\theta} \arrow[dashed]{dr}\\ & \P^1\times
	\P^1 \arrow{r} \arrow{d} & X(s_1)\\ & Z(s_1)\simeq X(s_1) \end{tikzcd}$$ and
	is $T$-equivariant for the diagonal $T$-action on $X(s_1)\times X(s_1)$.  The
	moment graph for $Z_{101}$ is 
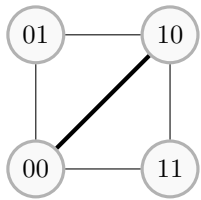
\begin{center}
		\begin{tikzpicture}[roundnode/.style={circle, draw=gray!60, fill=gray!5,
			very thick, minimum size=7mm}, ]
\node[roundnode]	(basepoint) {00}; \node[roundnode]	(upper) [above=of
			basepoint]		{01}; \node[roundnode] 	(right) [right=of basepoint]
			{11}; \node[roundnode] (upright) [above=of right]			{10};

\draw (upper.south) -- (basepoint.north); \draw (upper.east) -- (upright.west);
		\draw (basepoint.east) -- (right.west); \draw (upright.south) --
		(right.north); \draw[ultra thick] (basepoint.north east) -- (upright.south
		west); \end{tikzpicture} \end{center} where the bold line is an infinite
	family of $T$-stable curves (given by the images of the morphisms $s_t$).
	\end{example}

Before proving our main theorem, we prove a critical lemma.  \begin{lemma}
\label{isomorphism} Let $C\subset Z$ be an irreducible $T$-stable curve.
\begin{enumerate}[label=(\alph*)] \item If $C$ is not contained in a fiber of
			$\theta:Z\to G/B$, then $\theta|_C$ is an isomorphism.  \item If $C$ is
not a fiber of $\pi:Z\to Z'$, then $\pi|_C$ is an isomorphism.  \end{enumerate}
In particular, $C\simeq \P^1$.  \end{lemma} \begin{proof} We'll prove (a); the
	proof of (b) is similar. First, observe that $\theta|_C$ is an isomorphism if
	$C$ is a fiber of $\pi$, thus we may assume $C$ is not a fiber of $\pi$; let
	$C'=\pi(C)$. By induction on $\dim Z$, $\theta'|_{C'}$ is an isomorphism.
	There are two possibilities: \begin{enumerate} \item
				$(p_{\beta_k}\theta')|_{C'}$ is an isomorphism. In this case, the
				preimage $\pi^{-1}(C')$ is isomorphic (under $\theta$) to a surface in
				$G/B$.  In particular, $\theta|_C$ is an isomorphism.  \item
					$(p_{\beta_k}\theta')(C')$ is a $T$-fixed point in $G/P_{\beta_k}$.
					In this case, the preimage $\pi^{-1}(C')$ is $T$-equivariantly
					isomorphic to $\P^1\times \P^1$ with the diagonal $T$-action (coming
	from $G/B$). Hence, $C$ is a section of $\theta$ over $X(s_{\beta_k})$ by
	Lemma \ref{sections}. In particular, $\theta|_C$ is an isomorphism.
	\end{enumerate}

From (b), any $T$-stable curve $C$ is isomorphic to a fiber, and hence is
isomorphic to $\P^1$.  \end{proof}

\begin{lemma} \label{curve weight} Let $C\subset
Z(\alpha_1,\alpha_2,\ldots,\alpha_n)$ be a $T$-stable curve, and $k$ the
largest integer so that $\dim \pi_k(C) = 0$. Then, if $x\in C^T$, the weight at
$x$ is $\varepsilon_x(\alpha_k)$.  \end{lemma} \begin{proof} If $C$ is a fiber
of $\pi$, then the result follows since $\theta|_C$ is an isomorphism.
Otherwise, by Lemma \ref{isomorphism}, $\pi|_C$ is an isomorphism and the
result follows by induction on $n=\dim Z$.  \end{proof}

We are now ready to prove our main theorem, Theorem \ref{main theorem}, stated
in the introduction.

\begin{proof}[Proof of Theorem \ref{main theorem}.] For clarity, we will
	describe the moment graph pictures as Diagrams I-IV, and the
	characterizations in terms of $\varepsilon(\alpha_k)$ as Cases I-IV.

We only consider the case where $\pi(x),\pi(y)$ are joined by a $T$-stable
	curve $C$, since the other case is obvious. Let $\Sigma$ be defined as the
	pull-back $\Sigma = \pi^{-1}(C)$ as in the fiber diagram $$\begin{tikzcd}
		\Sigma \arrow[hook]{r} \arrow{d}[swap]{\pi|_{\Sigma}} & Z \arrow{d}{\pi}\\
	C \arrow[hook]{r} & Z' \end{tikzcd}$$ In particular, $\Sigma$ is the fiber
	product $$\begin{tikzcd} \Sigma \arrow{r}{\theta|_{\Sigma}}
		\arrow{d}[swap]{\pi|_{\Sigma}}	& G/B \arrow{d}{p_{\alpha_n}}\\ C
	\arrow{r}[swap]{p_{\alpha_n}\theta'|_C} & G/P_{\alpha_n} \end{tikzcd}$$

Lemma \ref{isomorphism} shows $p_{\alpha_n}\theta'|_C$ is either an isomorphism
	or a point map. If an isomorphism,  then $\Sigma$ is isomorphic (under
	$\theta$) to a surface in $G/B$. In particular, the moment graph restricted
	to $\Sigma$ is diagram I. 

Since the only curve whose homology class is not clear is the curve joining
	$\varepsilon_x^{\infty},\varepsilon_y^{\infty}$, we compute the homology
	class of that curve only. From Lemma \ref{curve weight}, the weight along $C$
	at $\pi(x)$ is $\varepsilon_{\pi(x)}(\alpha_k)$, and the weight at $\pi(y)$
	is $\varepsilon_{\pi(y)}(\alpha_k)$. Since $\theta|_{s(C)}=\theta'|_C$ is an
	isomorphism, the weights are preserved, and since $\theta(C)$ Is a translate
	of $X(s_{\alpha_k})$, we have $(\theta s)_*[C] = [X(s_{\alpha_k})]$. In
	particular, $w(\varepsilon_x^0)=w(\varepsilon_y^0)s_{\alpha_k}$. Since
	$p_{\alpha_n}\theta'|_C$ is an isomorphism, $\alpha_k\neq \alpha_n$, so case
	I holds.

We can then compute the relation between
	$w(\varepsilon_x^{\infty}),w(\varepsilon_y^{\infty})$: \begin{align*}
		w(\varepsilon_x^{\infty}) = w(\varepsilon_x^0)s_{\alpha_n} &=
		(w(\varepsilon_y^0)s_{\alpha_k})s_{\alpha_n}\\ &=
		(w(\varepsilon_y^0)s_{\alpha_n})s_{\beta} =
	w(\varepsilon_y^{\infty})s_{\beta} \end{align*} where $\beta =
	s_{\alpha_n}(\alpha_k) = \alpha_k - (\alpha_k,\alpha_n^{\vee})\alpha_n$.
	Thus, the degree of the image of the curve joining
	$\varepsilon_x^{\infty},\varepsilon_y^{\infty}$ is $[X(s_{\alpha_k})] -
	(\alpha_k,\alpha_n^{\vee})[X(s_{\alpha_n})]$.  Note, that since $\alpha_k\neq
	\alpha_n$, $(\alpha_k,\alpha_n^{\vee})\leq 0$. 

Moreover, the curve joining $\varepsilon_x^{\infty},\varepsilon_y^{\infty}$ has
	homology class $s_*[C] + ah$ for some constant $a$. Equating the push-forward
	calculations, we obtain $a=-(\alpha_k,\alpha_n^{\vee})$.  Therefore, the
	homology class of the curve joining
	$\varepsilon_x^{\infty},\varepsilon_y^{\infty}$ is $s_*[C] -
	(\alpha_k,\alpha_n^{\vee}) h.$

We now consider the case where $p_{\alpha_n}\theta'|_C$ is a point map.  In
	this case, $\Sigma$ is the trivial $\P^1$-bundle $\Sigma = C\times
	X(s_{\alpha_n})$. According to Lemma \ref{isomorphism}, there are cases
	corresponding to whether all $T$-fixed points in $\Sigma$ have distinct
	weights or not. 

As above, the weights along $C$ at $\varepsilon_x^0,\varepsilon_y^0$ are
	$\varepsilon_x^0(\alpha_k),\varepsilon_y^0(\alpha_k)$ respectively.  Thus,
	there are repeated weights in cases II, III, and IV. Indeed, there can only
	be exactly two fixed points with repeated weights, and they cannot be
	$\varepsilon_x^{\infty},\varepsilon_y^{\infty}$, otherwise there could not be
	a $T$-stable curve joining $\varepsilon_x^0,\varepsilon_y^0$. By Lemma
	\ref{sections}, there are infinitely many $T$-stable curves joining the two
	fixed points with repeated weights, hence we get the moment graph pictures
	which correspond to diagrams II, III, and IV. 

If all fixed points have distinct weights, we get the moment graph in diagram
I, and the curve joining $\varepsilon_x^{\infty},\varepsilon_y^{\infty}$ is a
fiber of the second projection. Therefore, if $w(\varepsilon_x^0) =
w(\varepsilon_y^0)$ and case I holds, the homology class of the unlabeled curve
is $s_*[C]$. 

If $\varepsilon_x^0,\varepsilon_y^0$ are the fixed points with repeated
weights, then the moment graph is given by diagram II. Moreover,
$w(\varepsilon_x^0)=w(\varepsilon_y^0)s_{\alpha_n}$ so case II holds.  From
this, $w(\varepsilon_x^0)=w(\varepsilon_y^{\infty})$ and vice-versa.
Therefore, the homology class of the diagonal curves is $s_*[C] - h$. 

If $w(\varepsilon_x^0)=w(\varepsilon_y^0)$, and there are two points with
repeated weights, they must be diagonally adjacent since both
$\varepsilon_x^0,\varepsilon_y^0$ cannot have repeated weights, in particular
cases III and IV hold. Diagrams III and IV are the cases where
$\varepsilon_x^0$ and $\varepsilon_y^0$ respectively have repeated weights. In
both cases, the diagonal family of curves have pushforward $[X(s_{\alpha_n})]$,
while $(\theta s)_*[C] = 0$. Therefore, the homology class of the diagonal
curves is $s_*[C]+h$.  \end{proof}





\begin{example} Consider the Bott-Samelson variety
	$Z=Z(\alpha_1,\alpha_2,\alpha_1)$ in type $A_2$. As stated in Section
	\ref{prelims}, the cone of effective 1-cycles is generated by the
	$B$-stable 1-cycles (see \cite[Lemma 2.1]{AndersonBottSamelson}.) A
	basis for the cone of effective 1-cycles is given by
$\beta_1=[Z_{010}]$, $\beta_2=[Z_{001}]$, $\beta_3=[Z_{100}]-[Z_{001}]$.  The
entire moment graph for the Bott-Samelson variety $Z$ is depicted in Figure
\ref{momentgraph picture}. \end{example}


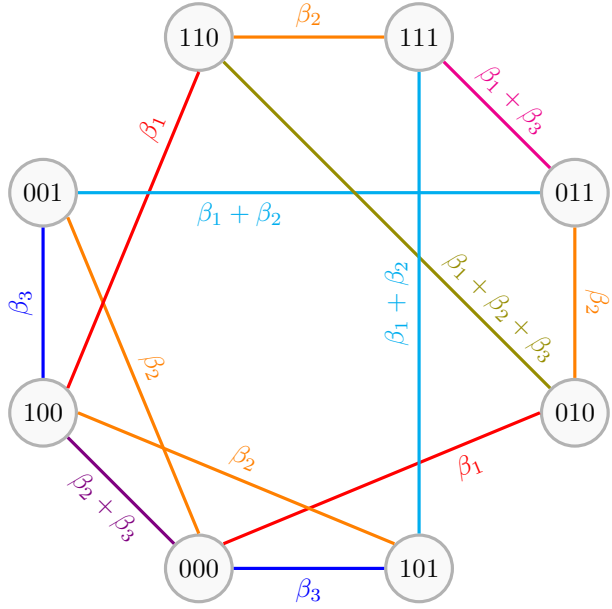
\begin{figure} \begin{center} \begin{tikzpicture}[node distance=2cm,
	roundnode/.style={circle, draw=gray!60, fill=gray!5, very thick, minimum
	size=7mm}, ]
\node[roundnode]	(zeroone) {000}; \node[roundnode]	(zerotwo)		[right = of
	zeroone] {101}; \node[roundnode]	(onezero)		[above left = of zeroone]
	{100}; \node[roundnode]	(onethree) [above right = of zerotwo]		{010};
	\node[roundnode] (twozero)		[above = of onezero]			{001};
	\node[roundnode]	(twothree)		[above = of onethree] {011}; \node[roundnode]
	(threeone)		[above right = of twozero]		{110}; \node[roundnode]
	(threetwo) [right = of threeone]			{111};

\draw[very thick, violet] 		(zeroone.north west) -- (onezero.south east)
	node[pos=0.5, below, sloped] {$\beta_2+\beta_3$}; \draw[very thick, orange]
	(zeroone.north) -- (twozero.south east)			node[pos=0.5, above, sloped]
	{$\beta_2$}; \draw[very thick, red] 		(zeroone.north east) --
	(onethree.west) node[pos=0.75, below, sloped] {$\beta_1$}; \draw[very thick,
	blue] (zeroone.east) -- (zerotwo.west)				node[pos=0.5, below,sloped]
	{$\beta_3$}; \draw[very thick, blue] 		(onezero.north) -- (twozero.south)
	node[pos=0.5, above, sloped] {$\beta_3$}; \draw[very thick, red]
	(onezero.north east) -- (threeone.south) node[pos=0.8, above, sloped]
	{$\beta_1$}; \draw[very thick, orange] (threeone.east) -- (threetwo.west)
	node[pos=0.5, above, sloped] {$\beta_2$}; \draw[very thick, cyan]
	(twozero.east) -- (twothree.west)				node[pos=0.35, below, sloped]
	{$\beta_1+\beta_2$}; \draw[very thick, magenta] 	(threetwo.south east) --
	(twothree.north west)		node[pos=0.5, above, sloped] {$\beta_1+\beta_3$};
	\draw[very thick, orange] (twothree.south) -- (onethree.north) node[pos=0.5,
	above, sloped] {$\beta_2$}; \draw[very thick, olive] (threeone.south east) --
	(onethree.north west)		node[pos=0.8, above, sloped]
	{$\beta_1+\beta_2+\beta_3$}; \draw[very thick, cyan] (zerotwo.north) --
	(threetwo.south) node[pos=0.5, above, sloped] {$\beta_1+\beta_2$}; \draw[very
	thick, orange]		(onezero.east) -- (zerotwo.north west) node[pos=0.5, above,
sloped] {$\beta_2$}; \end{tikzpicture} \end{center} \caption{The moment graph
for $Z(\alpha_1,\alpha_2,\alpha_1)$. \label{momentgraph picture}} \end{figure}

\section{Moduli space of stable maps} \label{moduli space}
We now turn to our computation of the quantum cohomology of the Bott-Samelson
variety $Z=Z(\alpha_1,\alpha_2,\alpha_1)$. In order to compute a presentation
$QH^*(Z)$, we need to compute many Gromov-Witten invariants. Some will vanish
using curve neighborhoods, however we will need to compute some nonzero
invariants to use brute-force calculations to finish the presentation. This
section provides the tools we need to compute these nonzero invariants. We
start with a definition.

\begin{defn} We say an effective curve class $\beta\in H_2(Z)$ is {\it
indecomposable} if $\beta$ cannot be expressed: $\beta=\beta_1+\beta_2$, where
$\beta_1,\beta_2\in H_2(Z)$ are effective.  \end{defn}

\begin{example} In the threefold $Z=Z(\alpha_1,\alpha_2,\alpha_1)$ in type
	$A_2$, the classes $[Z_{010}]$ and $[Z_{001}]$ are indecomposable, but
	$$[Z_{100}] = [Z_{001}] + \beta_3$$ where $\beta_3 = [C]$, the fiber of
	$\theta:Z\to F\ell(3)$ over the identity. 

Indeed, since $C$ is a fiber of $\theta$, $\theta_*[C] = 0$. Furthermore, $C$
	is a $T$-stable curve which is not contained in the fiber of
	$\pi:Z(\alpha_1,\alpha_2,\alpha_1)\to Z(\alpha_1,\alpha_2)$. So by Lemma
	\ref{isomorphism}, $\pi_*\beta_3 = [Z_{10}]$. Since $\pi_*\beta_3 =
\pi_*([Z_{100}] - [Z_{001}])$ and $\theta_*\beta_3 = \theta_*([Z_{100}] -
[Z_{001}])$, and $Z$ is a fiber product, we have $\beta_3=[Z_{100}] -
[Z_{001}]$ 

	In fact, the generators $\beta_1,\beta_2,\beta_3$ of the effective cone
	for $Z(\alpha_1,\alpha_2,\alpha_1)$ are all indecomposable.
\end{example}

Recall, the moduli space of stable maps $\overline{M}_{0,n}(Z,\beta)$ consists
of stable maps $f:C\to Z$, where $C$ is decorated with $n$ non-singular marked
points $p_1,\ldots,p_n\in C$, and $f_*[C]=\beta$. (The stability condition is
equivalent to the automorphism group of the marked curve $C$ being finite.) 

\begin{prop} \label{expected dimension} If $\beta\in H_2(Z)$ is indecomposable,
	and the Dynkin diagram of $G$ is simply laced, then the moduli space
	$\moduli{1}{Z}{\beta}$ is unobstructed; that is, $\moduli{1}{Z}{\beta}$ is
	smooth, irreducible, and has the expected dimension $$\dim
\moduli{1}{Z}{\beta} = \dim Z + \int_{\beta}c_1(T_Z) - 2.$$ \end{prop}
\begin{proof} The proof is by induction. Let $[g:C\to Z]\in
	\moduli{1}{Z}{\beta}$. Since $\beta$ is indecomposable and $C$ has a single
	marked point, $C\simeq \P^1$. There are two cases:

(a) $\beta$ is the class of the fiber of $\pi:Z\to Z'$. Since $g^*\pi^*T_{Z'}$
	is a trivial line bundle, $H^1(C,g^*\pi^*T_{Z'}) = 0$.  Moreover,
	$T_{\pi}=\theta^*T_{p_{\alpha_n}}$ and $\theta_*\beta = [X(s_{\alpha_n})]$
	which implies $g^*T_{\pi}\simeq \mathcal{O}_{\P^1}(2)$ (since
	$c_1(T_{p_{\alpha_n}})\cap [X(s_{\alpha_n})] = (\alpha_n,\alpha_n^{\vee})$).
	Thus $H^1(C,g^*T_{\pi}) = 0$, and therefore $H^1(C,g^*T_Z) = 0$.

(b) Otherwise, $\pi_*\beta\neq 0$ (and effective). Moreover, $\pi_*\beta$ is
	indecomposable, otherwise $\pi_*\beta = \beta_1' + \beta_2'$ which implies
	$\beta - s_*\beta_1'$ is effective. By induction,
	$\moduli{1}{Z'}{\pi_*\beta}$ is unobstructed, thus $H^1(C,g^*\pi^*T_{Z'})=0$.
To compute $H^1(C,g^*T_{\pi})$ there are two cases: \begin{enumerate} \item If
			$\theta_*\beta = 0$, then $g^*T_{\pi}$ is the trivial line bundle on $C$
			since $$\int_C c_1(g^*T_{\pi})\cap [C] = \int_Z c_1(T_{\pi})\cap \beta =
			\int_{G/B} c_1(T_{p_{\alpha_n}})\cap \theta_*\beta.$$ Thus,
			$H^1(C,g^*T_{\pi}) = 0$. 

\item If $\theta_*\beta \neq 0$, then since $\beta$ is represented by a
	$B$-stable curve, the image of that curve under $\theta$ is a Schubert curve.
		In particular, there is a unique $\alpha'$ for which the integral
		$$\int_{G/B}\theta_*\beta\cdot[Y(s_{\alpha'})] = \int_Z
		\beta\cdot\theta^*[Y(s_{\alpha'})] = \sum\int_Z\beta\cdot
		\sigma_{\varepsilon}$$ has value equal to 1, since $\beta$ is
		indecomposable. Therefore, $g^*T_{\pi}\simeq \mathcal{O}_{\P^1}(d)$ where
		$d=(\alpha,\alpha')\geq -1$ since the Dynkin diagram for $G$ is simply
laced, and so $H^1(C,g^*T_{\pi}) = 0$.  \end{enumerate} In either case, we have
	$H^1(C,g^*T_{\pi}) = 0$ and $H^1(C,g^*\pi^*T_{Z'}) = 0$, therefore
	$H^1(C,g^*T_Z) = 0$.  \end{proof}

\begin{cor}\label{fibercor} If $\beta$ is indecomposable and $Z$ is the
	disjoint union of curves of class $\beta$ (i.e. if $Z$ is a
	$\beta$-fibration), then $ev:\moduli{1}{Z}{\beta}\to Z$
is an isomorphism.  \end{cor} \begin{proof} Since $\beta$ is indecomposable,
	$\moduli{1}{Z}{\beta}$ is smooth. Moreover, every point in $Z$ lies on a
unique curve of class $\beta$, therefore $ev:\moduli{1}{Z}{\beta}\to Z$ is a
bijection. Since both are varieties over $\C$, $ev$ is an isomorphism.
\end{proof}

\begin{example}\label{beta3modulispace} For $Z=Z(\alpha_1,\alpha_2,\alpha_1)$
	and $\beta=\beta_3$ $$ev:\moduli{1}{Z}{\beta}\to Z_{101}$$ is an isomorphism
	from Corollary \ref{fibercor}. 

If $h$ is the class of the fiber for any Bott-Samelson variety $Z$,
$ev:\moduli{1}{Z}{h}\to Z$ is an isomorphism since $h$ is indecomposable.
\end{example}


\section{Curve neighborhoods} \label{curve nbhds}
As we stated in the introduction, the connection between our moment graph
	result and quantum cohomology is given by curve neighborhoods.  Background on
	curve neighborhoods, particularly in the context of homogeneous spaces, can
	be found in \cite{curvenbhds}. The main difference here is that curve
	neighborhoods need not ``grow.''

\begin{defn} Let $X$ be a variety, $Y$ a subset of $X$, and $\beta\in A_1(X)$
	an effective curve class. The curve neighborhood $\Gamma_{\beta}(Y)$ is
	defined by $$\Gamma_{\beta}(Y) := ev_1(ev_2^{-1}(Y))$$ where
	$ev_i:\moduli{2}{X}{\beta}\to X$ ($i=1,2$) are the evaluation morphisms;
	$\Gamma_{\beta}(Y)$ is given the reduced scheme structure.

Observe, if $X$ is a $G$-variety for an algebraic group $G$, and $Y$ is a
	$G$-stable closed subvariety of $X$, then $\Gamma_{\beta}(Y)$ is a $G$-stable
	closed subvariety of $X$. Also, note that since the morphism
	$ev:\moduli{2}{X}{\beta}\to X$ is proper, $\Gamma_{\beta}(Y)$ can be realized
as the closure of the union of all curves $C$ of class $\beta$ passing through
$Y$.  \end{defn}

\begin{prop} For any $B$-stable subvariety $\Omega\subset
	Z(\alpha_1,\alpha_2,\alpha_1)$, the curve neighborhood
	$\Gamma_{\beta_3}(\Omega)\subset Z_{101}$. Furthermore, if $\theta(\Omega) =
	X(s_1)$, then $\Gamma_{\beta_3}(\Omega) = Z_{101}$.  \end{prop} \begin{proof}
		Note that any curve of class $\beta_3$ collapses under the map
		$\theta:Z(\alpha_1,\alpha_2,\alpha_1)\to G/B$. Since $\theta$ is an
		isomorphism outside the locally closed set $Z_{101}^{\circ}$, any curve of
		class $\beta_3$ is contained in the sub Bott-Samelson variety $Z_{101}$.
		Since $Z_{101}$ is a closed, $B$-stable variety in
		$Z(\alpha_1,\alpha_2,\alpha_1)$, the first claim follows. 

For the second claim, since $Z_{101}\simeq \P^1\times \P^1$ where $\beta_3$ is
the class of the fiber of the second projection, every point $x\in Z_{101}$ has
a unique curve of class $\beta_3$ passing through it which also intersects
$\Omega$.  \end{proof}

\begin{lemma} \label{curvenbhd example} The following are curve neighborhoods
	for the Bott-Samelson variety $Z(\alpha_1,\alpha_2,\alpha_1)$:
	$$\Gamma_{\beta_3}(Z_{100}) = Z_{101},\quad \Gamma_{\beta_3}(Z_{010}) =
	\theta^{-1}(x_e),\quad \Gamma_{\beta_3}(Z_{001}) = Z_{101},\quad
\Gamma_{\beta_1}(Z_{100}) = Z_{110}.$$ \end{lemma} \begin{proof} We prove each
	equality in the order specified above.  \begin{enumerate}[labelindent=0em,
			labelwidth=2cm, labelsep*=1em, leftmargin=!]
		\item[$\Gamma_{\beta_3}(Z_{100})$:] Note that $Z_{100}\subset Z_{101}$.  In
			fact, the image of $Z_{100}$ under $\theta: Z(\alpha_1,\alpha_1)\to
			X(s_{\alpha_1})$ is $X(s_{\alpha_1})$.  Therefore, the union of the
			fibers of $pr_2: Z(\alpha_1,\alpha_1)\to X(s_{\alpha_1})$ is the curve
			neighborhood $\Gamma_{\beta_3}(Z_{100})$. However this union is the whole
			space. Therefore, $\Gamma_{\beta_3}(Z_{100})= Z_{101}$. 

\item[$\Gamma_{\beta_3}(Z_{010})$:] Since $Z_{010}\cap Z_{101} = x_{000}$, and
	since there is a unique curve of class $\beta_3$ passing through that point
			(the curve joining $x_{000}$ to $x_{101}$) the curve neighborhood is
			$\Gamma_{\beta_3}(Z_{010}) = C$, the fiber of $\theta$ over $x_e\in G/B$.

\item[$\Gamma_{\beta_3}(Z_{001})$:] A similar analysis to
	$\Gamma_{\beta_3}(Z_{100})$ shows that $\Gamma_{\beta_3}(Z_{001}) = Z_{101}$
			(since the image of $Z_{001}$ under $\theta:Z(\alpha_1,\alpha_1)\to
			X(s_{\alpha_1})$ is $X(s_{\alpha_1})$).

\item[$\Gamma_{\beta_1}(Z_{100})$:] There is a $T$-stable curve of class
	$\beta_1$ joining $x_{100}$ to $x_{110}$. Since the $B$-action on the cell
			$Z_{110}^{\circ}$ is transitive ($\theta|_{Z_{110}}$ is an isomorphism,)
			and since $Z_{100}$ is $B$-stable, we see that
	$\Gamma_{\beta_1}(Z_{100})=Z_{110}$.  \end{enumerate} \end{proof}

\begin{remark} In general, it is clear that curve neighborhoods will be
	connected. However, it would be useful to have a condition for when a
	particular curve neighborhood is irreducible. For example, for Schubert
	varieties, it is known that all such curve neighborhoods are irreducible (see
\cite{BCMP}).  \end{remark}

\section{Quantum cohomology} \label{quantum}
Let $X$ be a smooth, projective $\C$-variety. Fix a homogeneous basis
	$\{\gamma_j\}$ for $H^*(X)$, and a basis $\{\beta_k\}$ for the cone of
	effective curves in $H_2(X)$. As a $\Q$-vector space, the (small) quantum
	cohomology $QH^*(X)= H^*(X;\Q)\otimes \Q[q^{\beta_k}]$, where there is one
	quantum parameter $q^{\beta_k}$ for each generator of the effective cone. The
	quantum product of two classes $x,y\in H^*(X)$ is defined by $$x\ast y =
	\sum_{\beta,j} I_{\beta}(x,y,\gamma_j^{\vee}) q^{\beta}\gamma_j$$ where
	$\gamma_j^{\vee}$ is Poincar\'{e} dual to $\gamma_j$, and the {\it
	Gromov-Witten invariant} $I_{\beta}(x,y,\gamma_j^{\vee})$ is defined by
	$$I_{\beta}(x,y,\gamma_j^{\vee}) = \int_{[\moduli{3}{X}{\beta}]^{\virt}}
	ev_1^*(x)\cdot ev_2^*(y)\cdot ev_3^*(\gamma_j^{\vee})$$ where
	$[\moduli{3}{X}{\beta}]^{\virt}$ is the {\it virtual fundamental class}, a
	generalization of the fundamental class which is necessary since
	$\moduli{3}{X}{\beta}$ is not generally irreducible or even equidimensional.
	The class $[\moduli{3}{X}{\beta}]^{\virt}$ has the ``expected dimension'' of
	the moduli space (see Proposition \ref{expected dimension}).

In general, for $\sigma_1,\ldots,\sigma_n\in H^*(X)$, Gromov-Witten invariants
	are defined on $\moduli{n}{X}{\beta}$ as
	$$I_{\beta}(\sigma_1,\ldots,\sigma_n) = \int_{[\moduli{n}{X}{\beta}]^{\virt}}
	\prod_j ev_j^*(\sigma_j).$$

For us, there are two important properties of Gromov-Witten invariants:
\begin{enumerate} \item The Gromov-Witten invariant
			$I_{\beta}(\sigma_1,\ldots,\sigma_n) = 0$ unless $$\sum_j \codim \sigma_j
			= \dim X + \int_{\beta} c_1(T_X) + n - 3;$$ this property is called the
			{\it codimension condition}. 

\item If $\sigma_1 \in H^2(X)$ (i.e. if $\sigma_1$ is a divisor class,) then
$$I_{\beta}(\sigma_1,\sigma_2,\ldots,\sigma_n) = \left(\int_{\beta}
\sigma_1\right)I_{\beta}(\sigma_2,\ldots,\sigma_n);$$ this property is called
the {\it divisor axiom}.  \end{enumerate}

As a result of the codimension condition, $QH^*(X)$ is equipped with a grading
compatible with the grading on $H^*(X)$: $$\deg q^{\beta} = \int_{\beta}
c_1(T_X).$$

\begin{lemma} For $Z=Z(\alpha_1,\alpha_2,\alpha_1)$, the degrees of the quantum
parameters are $$\deg q^{\beta_1} = 1,\quad \deg q^{\beta_2} = 2,\quad \deg
q^{\beta_3} = 1.$$ \end{lemma} \begin{proof} Let $Z'$ denote the Bott-Samelson
	variety $Z(\alpha_1,\alpha_2)$. From the exact sequence of tangent bundles on
	$Z$ $$\begin{tikzcd} 0 \arrow{r} & T_{\pi} \arrow{r} & T_Z \arrow{r} &
	\pi^*T_{Z'} \arrow{r} & 0 \end{tikzcd}$$ we get $c_1(T_Z) = c_1(T_{\pi}) +
	c_1(\pi^*T_{Z'})$, where $T_{\pi}$ denotes the relative tangent bundle.
	Therefore, from the projection formula \begin{align*} \int_{\beta} c_1(T_Z)
		&= \int_Z c_1(T_Z)\cdot \beta\\ &= \int_{G/B} c_1(T_{p_{\alpha_1}}) \cdot
		\theta_*\beta + \int_{Z'} c_1(T_{Z'})\cdot \pi_*\beta.  \end{align*}

For $\beta = \beta_2$, since $\beta_2 = \pi^*[pt]$ the second integral
	vanishes. Thus $$\deg q^{\beta_2} = \int_{\beta_2} c_1(T_Z) = \int_{G/B}
	c_1(T_{p_{\alpha_1}})\cdot [X(s_{\alpha_1})] = (\alpha_1,\alpha_1^{\vee}) =
	2.$$

For $\beta = \beta_3$, the pushforward $\theta_*\beta_3 = 0$, which implies the
	first integral vanishes. Thus \begin{align*} \deg q^{\beta_3} =
		\int_{\beta_3}c_1(T_Z) &= \int_{Z'} c_1(T_{Z'})\cdot [Z'_{10}]\\ &=
		\int_{G/B} c_1(T_{p_{\alpha_2}})\cdot [X(s_{\alpha_1})] + \int_{\P^1}
	c_1(T_{\P^1}) \cdot [pt]\\ &= (\alpha_2,\alpha_1^{\vee}) + 2 = 1.
	\end{align*}

For $\beta=\beta_1$, we have \begin{align*} \deg q^{\beta_1} =
\int_{\beta_1}c_1(T_Z) &= \int_{G/B} c_1(T_{p_{\alpha_1}})\cdot
[X(s_{\alpha_2})] + \int_{Z'} c_1(T_{Z'})\cdot [Z'_{01}]\\ &=
(\alpha_1,\alpha_2^{\vee}) + 2 = 1.  \end{align*} \end{proof}

	\begin{remark} \label{Fano}
	In the course of proving the previous Lemma, we showed $c_1(T_Z) =
	3\sigma_{100} + \sigma_{010} + 2\sigma_{001}$.		

	In \cite{BSample}, a basis for the ample cone on a Bott-Samelson
	variety is given; the $\mathcal{O}_{\w}(1)$ basis. These line bundles are
	pullbacks of the line bundles $L_{\omega_{\alpha_k}}$ corresponding to
	a dominant fundamental weight $\omega_{\alpha_k}$. Since
	$c_1(L_{\omega_{\alpha_k}}) = [Y(s_{\alpha_k})]$, the line bundle
	$\mathcal{O}_{\w}(1)$ has Chern class $\sigma_{100}+\sigma_{001}$ (for
	$\w = (1,2,1)$). The other two generators have Chern classes
	\begin{align*}
		c_1(\mathcal{O}_{1,2}(1)) &= \sigma_{010}\\
		c_1(\mathcal{O}_1(1)) &= \sigma_{100}
	\end{align*}
	Therefore, we can write
	$$c_1(T_Z) = c_1(\mathcal{O}_1(1)) + c_1(\mathcal{O}_{1,2}(1)) +
	2c_1(\mathcal{O}_{1,2,1}(1)).$$
	Thus, $-K_Z$ is ample, that is $Z$ is Fano.
\end{remark}

From here on, $Z$ denotes the Bott-Samelson variety
	$Z(\alpha_1,\alpha_2,\alpha_1)$ in type $A_2$. In order to compute a
	presentation for the small quantum cohomology $QH^*(Z)$, it suffices to
	quantize the relations in the ordinary cohomology $H^*(Z)$; see
	\cite[Proposition 11]{FultonPandharipande} for more details. Since the
	relations in $H^*(Z)$ are all in codimension one, and since it is necessary
	for the approach we use, we will compute all products $\sigma_{(j)}\ast
	\sigma_{\varepsilon}$. We organize this data in what we call Chevalley
	matrices. 

Since one of the terms of our quantum products will always be a divisor class,
	the divisor axiom will always be used to reduced three-point Gromov-Witten
	invariants to two-point invariants. For two-point invariants, we have the
	following lemma which relates curve neighborhoods to the vanishing of
	Gromov-Witten invariants.

\begin{lemma} \label{curve nbhd technique} Let $X$ be a smooth, projective
	$\C$-variety, $\Omega\subset X$ a closed subvariety, $\gamma\in H^*(X)$, and
	$\beta\in H_2(X)$ an effective curve class. Suppose $$\codim \gamma + \codim
	[\Omega] = \dim X + \int_{\beta} c_1(T_X) - 1$$ (that is, the codimension
	condition is satisfied.) Denote the irreducible components of
	$\Gamma_{\beta}(\Omega)$ by $\Gamma_i$, that is: $$\Gamma_{\beta}(\Omega) =
	\Gamma_1 \cup \Gamma_2 \cup \cdots \cup \Gamma_k.$$ \begin{enumerate} \item
			If $\dim \Gamma_{\beta}(\Omega) < \codim \gamma$, then
		$\GW{\gamma,[\Omega]}{\beta} = 0$.  \item If $\dim \Gamma_{\beta}(\Omega) =
			\codim \gamma$ and $\int_X \gamma\cdot [\Gamma_i] = 0$ for each $1\leq i
			\leq k$, then $\GW{\gamma,[\Omega]}{\beta} = 0$.  \end{enumerate}
\end{lemma} \begin{proof} Using the projection formula, we have \begin{align*}
	\GW{\gamma,[\Omega]}{\beta} &= \int_{[\moduli{2}{X}{\beta}]^{\virt}]}
	ev_1^*(\gamma)\cdot ev_2^*[\Omega]\\ &= \int_X \gamma\cdot
	ev_{1\ast}(ev_2^*[\Omega]\cdot [\moduli{2}{X}{\beta}]^{\virt}) \end{align*}
	where, by definition of curve neighborhoods, $$ev_{1\ast}(ev_2^*[\Omega]\cdot
	[\moduli{2}{X}{\beta}]^{\virt}) = \sum_{i=1}^k m_i\>[\Gamma_i];$$ the
	constants $m_i\geq 0$ are zero when $\dim \Gamma_i < \dim ev_2^*(\Omega)$,
	otherwise they are the degrees of $ev_1$ restriced to each irreducible
	component. 

The desired conclusion follows in both cases since $\int_X \gamma\cdot
[\Gamma_i] = 0$ for all $i$.  \end{proof}

Using the moment graph to compute curve neighborhoods for
$Z(\alpha_1,\alpha_2,\alpha_1)$, and Lemma \ref{curve nbhd technique}, we are
able to show that certain Gromov-Witten invariants vanish.  Combining this with
the explicit moduli space results obtained in Section \ref{moduli space}, we
are able to compute many of the Gromov-Witten invariants needed to compute
$QH^*(Z)$. However, there are still some unknown invariants that are needed.

Using a computer, and the fact that $QH^*(Z)$ is a commutative ring, we are
able to solve for all of the unknowns in the Chevalley matrices. In each of the
following subsections, we record the calculations necessary to produce the
Chevalley matrices, the matrix $A$ corresponding to quantum multiplication by
$\sigma_{100}$, $B$ corresponding to quantum multiplication by $\sigma_{010}$,
and $C$ corresponding to quantum multiplication by $\sigma_{001}$.

\subsection{Chevalley matrix $A$} Given any $\varepsilon\in \{0,1\}^3$, we have
\begin{equation*} \sigma_{100}\ast\sigma_{\varepsilon} =
	\sum_{\beta,\varepsilon'}\GW{\sigma_{100},\sigma_{\varepsilon},[Z_{\varepsilon'}]}{\beta}\>q^{\beta}\sigma_{\varepsilon'}
\end{equation*} We will calculate the quantum product
$\sigma_{100}\ast\sigma_{\varepsilon}$ for each $\varepsilon\in \{0,1\}^3$.
\begin{itemize} \item For $\varepsilon = 000$, since $1 = [Z] \in H^*(Z)$ is
			also the identity in $QH^*(Z)$ we have \begin{equation} \sigma_{100}\ast
				1 = \sigma_{100} \end{equation}

\item For $\varepsilon = 100$, $010$, or $001$, we can use the divisor axiom
	twice to reduce the three-point Gromov-Witten invariants to one-point
		invariants. 

Using the codimension condition, for the curve classes $\beta\neq 0$ which need
		to be considered, $q^{\beta}$ is either degree 1 or 2.  Therefore, the
		curve classes $\beta$ for which
		$\GW{\sigma_{100},\sigma_{\varepsilon},[Z_{\varepsilon'}]}{\beta}$ is
		possibly nonzero are $$\beta_3, \beta_1+\beta_3, 2\beta_3$$ Corollary
		\ref{fibercor} implies $ev:\moduli{1}{Z}{\beta_3}\to Z_{101}$ is an
		isomorphism.  Thus \begin{align*} \GW{[Z_{100}]}{\beta_3} &=
		\int_{Z_{101}}[Z_{100}] = -1\\ \GW{[Z_{010}]}{\beta_3} &=
		\int_{Z_{101}}[Z_{010}] = 1\\ \GW{[Z_{001}]}{\beta_3} &=
		\int_{Z_{101}}[Z_{001}] = 0 \end{align*} We will assign variables for the
		remaining Gromov-Witten invariants; we compute the values of these unknowns
		using brute force after analyzing all three Chevalley matrices.
		$$\GW{[pt]}{\beta_1+\beta_3} = x_1,\quad \GW{[pt]}{2\beta_3} = x_2$$

We can now write the quantum products as follows: \begin{align}
	\sigma_{100}\ast \sigma_{100} &= -q_3 \sigma_{100} + q_3 \sigma_{010} +
	(q_1q_3 x_1 + 4 q_3^2 x_2)\\ \sigma_{100}\ast \sigma_{010} &= \sigma_{110} +
	q_1q_3 x_1\\ \sigma_{100}\ast \sigma_{001} &= \sigma_{101} + q_3 \sigma_{100}
- q_3 \sigma_{010} + (-q_1q_3 x_1 - 4 q_3^2 x_2) \end{align}

\item For $\varepsilon = 110, 101, 011$, we can only use the divisor axiom once
	to reduce the three-point Gromov-Witten invariant to a two-point invariant.
		Moreover, using the codimension condition and the divisor axiom, we can
		reduce the curve classes $\beta$ which need to be considered to the
		following:
		$$\beta_3,\beta_1+\beta_3,2\beta_3,2\beta_1+\beta_3,\beta_1+2\beta_3,\beta_2+\beta_3,3\beta_3$$
		From the codimension condition, if $\ell(\varepsilon) = 2$,
		$\GW{\sigma_{100},\sigma_{\varepsilon},[Z_{\varepsilon'}]}{\beta_3} = 0$
		unless $\ell(\varepsilon') = 2$. Moreover, the curve neighborhoods for
		class $\beta_3$ have been considered in Section \ref{curve nbhds}:
		$$\Gamma_{\beta_3}(Z_{110}) = Z_{101},\quad \Gamma_{\beta_3}(Z_{101}) =
		Z_{101},\quad \Gamma_{\beta_3}(Z_{011}) = Z_{101}$$ Using Lemma \ref{curve
		nbhd technique}, the curve neighborhoods show
		$\GW{\sigma_{\varepsilon},[Z_{\varepsilon'}]}{\beta_3} = 0$ unless
		$\varepsilon = 101$. Thus we have shown that six Gromov-Witten invariants
		vanish, leaving three more (for curve class $\beta_3$) which are unknown.

We assign variables for the remaining Gromov-Witten invariants.
		$$\begin{array}{ccc} \GW{\sigma_{101},[Z_{110}]}{\beta_3} = x_3, &
			\GW{\sigma_{101},[Z_{101}]}{\beta_3} = x_4, &
			\GW{\sigma_{101},[Z_{011}]}{\beta_3} = x_5,\\
			\GW{\sigma_{110},[Z_{100}]}{\beta_1+\beta_3} = x_6,	&
			\GW{\sigma_{110},[Z_{010}]}{\beta_1+\beta_3} = x_7, &
			\GW{\sigma_{110},[Z_{001}]}{\beta_1+\beta_3} = x_8,\\
			\GW{\sigma_{101},[Z_{100}]}{\beta_1+\beta_3} = x_9,	&
			\GW{\sigma_{101},[Z_{010}]}{\beta_1+\beta_3} = x_{10},	&
			\GW{\sigma_{101},[Z_{001}]}{\beta_1+\beta_3} = x_{11},\\
			\GW{\sigma_{011},[Z_{100}]}{\beta_1+\beta_3} = x_{12},	&
			\GW{\sigma_{011},[Z_{010}]}{\beta_1+\beta_3} = x_{13},	&
			\GW{\sigma_{011},[Z_{001}]}{\beta_1+\beta_3} = x_{14},\\
			\GW{\sigma_{110},[Z_{100}]}{2\beta_3} = x_{15},		&
			\GW{\sigma_{110},[Z_{010}]}{2\beta_3} = x_{16}, &
			\GW{\sigma_{110},[Z_{001}]}{2\beta_3} = x_{17},\\
			\GW{\sigma_{101},[Z_{100}]}{2\beta_3} = x_{18},		&
			\GW{\sigma_{101},[Z_{010}]}{2\beta_3} = x_{19}, &
			\GW{\sigma_{101},[Z_{001}]}{2\beta_3} = x_{20},\\
			\GW{\sigma_{011},[Z_{100}]}{2\beta_3} = x_{21},		&
			\GW{\sigma_{011},[Z_{010}]}{2\beta_3} = x_{22}, &
			\GW{\sigma_{011},[Z_{001}]}{2\beta_3} = x_{23},\\
			\GW{\sigma_{110},[pt]}{2\beta_1+\beta_3} = x_{24},		&
			\GW{\sigma_{101},[pt]}{2\beta_1+\beta_3} = x_{25},		&
			\GW{\sigma_{011},[pt]}{2\beta_1+\beta_3} = x_{26},\\
			\GW{\sigma_{110},[pt]}{\beta_1+2\beta_3} = x_{27}, &
			\GW{\sigma_{101},[pt]}{\beta_1+2\beta_3} = x_{28}, &
			\GW{\sigma_{011},[pt]}{\beta_1+2\beta_3} = x_{29},\\
			\GW{\sigma_{110},[pt]}{\beta_2+\beta_3} = x_{30},		&
			\GW{\sigma_{101},[pt]}{\beta_2+\beta_3} = x_{31},			&
			\GW{\sigma_{011},[pt]}{\beta_2+\beta_3} = x_{32},\\
			\GW{\sigma_{110},[pt]}{3\beta_3} = x_{33},			&
			\GW{\sigma_{101},[pt]}{3\beta_3} = x_{34}, &
		\GW{\sigma_{011},[pt]}{3\beta_3} = x_{35} \end{array}$$

We now record these quantum products: \begin{align}
	\sigma_{100}\ast\sigma_{110} 	&= (q_1q_3 x_6 + 2 q_3^2 x_{15})\sigma_{100} +
	(q_1q_3 x_7 + 2 q_3^2 x_{16})\sigma_{010}\\ \nonumber				&+ (q_1q_3 x_8 +
	2 q_3^2 x_{17})\sigma_{001} + (q_1^2q_3 x_{24} + 2 q_1q_3^2 x_{27} + q_2q_3
	x_{30} + 3 q_3^3 x_{33})\\ \sigma_{100}\ast \sigma_{101} 	&= q_3 x_3
	\sigma_{110} + q_3 x_4 \sigma_{101} + q_3 x_5 \sigma_{011} + (q_1q_3 x_9 + 2
	q_3^2 x_{18})\sigma_{100}\\ \nonumber &+ (q_1q_3 x_{10} + 2 q_3^2
	x_{19})\sigma_{010} + (q_1q_3 x_{11} + 2 q_3^2 x_{20})\sigma_{001}\\
	\nonumber				&+ q_1^2q_3 x_{25} + 2 q_1q_3^2 x_{28} + q_2q_3 x_{31} + 3
	q_3^3 x_{34}\\ \sigma_{100}\ast \sigma_{011}	&= [pt] + (q_1q_3 x_{12} + 2
	q_3^2 x_{21})\sigma_{100} + (q_1q_3 x_{13} + 2 q_3^2 x_{22})\sigma_{010} +
(q_1q_3 x_{14} + 2 q_3^2 x_{23})\sigma_{001}\\ \nonumber &+ q_1^2q_3 x_{26} + 2
q_1q_3^2 x_{29} + q_2q_3 x_{32} + 3 q_3^3 x_{35} \end{align}

\item For $\varepsilon = 111$, the curve classes which possibly contribute with
	non-zero Gromov-Witten invariants are \begin{align*}
		\beta_3,\beta_1+\beta_3,2\beta_3,2\beta_1+\beta_3,\beta_1+2\beta_3,\beta_2+\beta_3,3\beta_3,\\
		3\beta_1+\beta_3,2\beta_1+2\beta_3, \beta_1+3\beta_3, \beta_2+2\beta_3,
	4\beta_3, \beta_1+\beta_2+\beta_3 \end{align*} Using the fundamental class
	axiom for Gromov-Witten invariants, $$\GW{\sigma_{100},[pt],[Z]}{\beta_3} =
	0.$$ We assign variables for the remaining Gromov-Witten invariants.
	$$\begin{array}{ccc} \GW{[pt],[Z_{110}]}{\beta_1+\beta_3} = x_{36}, &
		\GW{[pt],[Z_{101}]}{\beta_1+\beta_3} = x_{37}, &
		\GW{[pt],[Z_{011}]}{\beta_1+\beta_3} = x_{38},\\
		\GW{[pt],[Z_{110}]}{2\beta_3} = x_{39},			& \GW{[pt],[Z_{101}]}{2\beta_3}
		= x_{40}, & \GW{[pt],[Z_{011}]}{2\beta_3} = x_{41},\\
		\GW{[pt],[Z_{100}]}{2\beta_1+\beta_3} = x_{42},	&
		\GW{[pt],[Z_{010}]}{2\beta_1+\beta_3} = x_{43},		&
		\GW{[pt],[Z_{001}]}{2\beta_1+\beta_3} = x_{44},\\
		\GW{[pt],[Z_{100}]}{\beta_1+2\beta_3} = x_{45},	&
		\GW{[pt],[Z_{010}]}{\beta_1+2\beta_3} = x_{46},		&
		\GW{[pt],[Z_{001}]}{\beta_1+2\beta_3} = x_{47},\\
		\GW{[pt],[Z_{100}]}{\beta_2+\beta_3} = x_{48},		&
		\GW{[pt],[Z_{010}]}{\beta_2+\beta_3} = x_{49},		&
		\GW{[pt],[Z_{001}]}{\beta_2+\beta_3} = x_{50},\\
		\GW{[pt],[Z_{100}]}{3\beta_3} = x_{51},			& \GW{[pt],[Z_{010}]}{3\beta_3}
		= x_{52}, & \GW{[pt],[Z_{001}]}{3\beta_3} = x_{53},\\
		\GW{[pt],[pt]}{3\beta_1+\beta_3} = x_{54}, &
		\GW{[pt],[pt]}{2\beta_1+2\beta_3} = x_{55}, &
		\GW{[pt],[pt]}{\beta_1+3\beta_3} = x_{56},\\
	\GW{[pt],[pt]}{\beta_2+2\beta_3} = x_{57}, & \GW{[pt],[pt]}{4\beta_3} =
	x_{58}, & \GW{[pt],[pt]}{\beta_1+\beta_2+\beta_3} = x_{59} \end{array}$$

We now record the final quantum product for the Chevalley matrix $A$:
\begin{align} \sigma_{100}\ast [pt]	&= (q_1q_3 x_{36} + 2 q_3^2
	x_{39})\sigma_{110} + (q_1q_3 x_{37} + 2 q_3^2 x_{40})\sigma_{101} + (q_1q_3
	x_{38} + 2 q_3^2 x_{41})\sigma_{011}\\ \nonumber &+ (q_1^2q_3 x_{42} + 2
	q_1q_3^2 x_{45} + q_2q_3 x_{48} + 3 q_3^3 x_{51})\sigma_{100}\\ \nonumber &+
	(q_1^2q_3 x_{43} + 2 q_1q_3^2 x_{46} + q_2q_3 x_{49} + 3 q_3^3
	x_{52})\sigma_{010}\\ \nonumber		&+ (q_1^2q_3 x_{44} + 2 q_1q_3^2 x_{47} +
	q_2q_3 x_{50} + 3 q_3^3 x_{53})\sigma_{010}\\ \nonumber		&+ q_1^3q_3 x_{54}
+ 2 q_1^2q_3^2 x_{55} + 3 q_1q_3^3 x_{56} + 2 q_2q_3^2 x_{57} + 4 q_3^4 x_{58}
+ q_1q_2q_3 x_{59} \end{align} \end{itemize}

\subsection{Chevalley matrix $B$} As in the previous section, we can write the
quantum product as $$\sigma_{010}\ast \sigma_{\varepsilon} =
\sum_{\beta,\varepsilon'}
\GW{\sigma_{010},\sigma_{\varepsilon},[Z_{\varepsilon'}]}{\beta}\> q^{\beta}
\sigma_{\varepsilon'}$$

\begin{itemize} \item As before, $1 = [Z]\in H^*(Z)$ is also the identity in
			$QH^*(Z)$. Therefore, \begin{equation} \sigma_{010}\ast 1 = \sigma_{010}
			\end{equation}

\item For $\varepsilon = 100,010,001$, we can use the divisor axiom twice to
	reduce the three-point Gromov-Witten invariants to one-point invariants. The
		curve classes which give possibly non-zero Gromov-Witten invariants are
		$$\beta_1, \beta_1+\beta_3, 2\beta_1$$ We previously defined
		$\GW{[pt]}{\beta_1+\beta_3} = x_1$, so we only need the following new
		invariants: $$\begin{array}{cccc} \GW{[Z_{100}]}{\beta_1} = y_1, &
			\GW{[Z_{010}]}{\beta_1} = y_2,		& \GW{[Z_{001}]}{\beta_1} = y_3,		&
		\GW{[pt]}{2\beta_1} = y_4 \end{array}$$ We can then write the quantum
		products as \begin{align} \sigma_{010}\ast \sigma_{100} &= \sigma_{110} +
			q_1q_3 x_1\\ \sigma_{010}\ast \sigma_{010} &= \sigma_{110} + q_1 y_1
		\sigma_{100} + q_1 y_2 \sigma_{010} + q_1 y_3 \sigma_{001} + (q_1q_3 x_1 +
		2 q_1^2 y_4)\\ \sigma_{010}\ast \sigma_{001} &= \sigma_{011} - q_1q_3 x_1
		\end{align}

\item For $\varepsilon = 110,101,011$, we use the divisor axiom to reduce the
	three-point Gromov-Witten invariants to two-point invariants. The curve
		classes which give possibly non-zero Gromov-Witten invariants are
		$$\beta_1,\beta_1+\beta_3,2\beta_1,2\beta_1+\beta_3,\beta_1+2\beta_3,\beta_1+\beta_2,3\beta_1$$
		We assign variables for the remaining Gromov-Witten invariants (note: many
		of these Gromov-Witten invariants were already considered in the Chevalley
		matrix $A$.) $$\begin{array}{ccc} \GW{\sigma_{110},[Z_{110}]}{\beta_1} =
			y_5, & \GW{\sigma_{110},[Z_{101}]}{\beta_1} = y_6, &
			\GW{\sigma_{110},[Z_{011}]}{\beta_1} = y_7,\\
			\GW{\sigma_{101},[Z_{110}]}{\beta_1} = y_8, &
			\GW{\sigma_{101},[Z_{101}]}{\beta_1} = y_9, &
			\GW{\sigma_{101},[Z_{011}]}{\beta_1} = y_{10},\\
			\GW{\sigma_{011},[Z_{110}]}{\beta_1} = y_{11}, &
			\GW{\sigma_{011},[Z_{101}]}{\beta_1} = y_{12}, &
			\GW{\sigma_{011},[Z_{011}]}{\beta_1} = y_{13},\\
			\GW{\sigma_{110},[Z_{100}]}{\beta_1+\beta_3} = x_6, &
			\GW{\sigma_{110},[Z_{010}]}{\beta_1+\beta_3} = x_7, &
			\GW{\sigma_{110},[Z_{001}]}{\beta_1+\beta_3} = x_8,\\
			\GW{\sigma_{101},[Z_{100}]}{\beta_1+\beta_3} = x_9, &
			\GW{\sigma_{101},[Z_{010}]}{\beta_1+\beta_3} = x_{10},	&
			\GW{\sigma_{101},[Z_{001}]}{\beta_1+\beta_3} = x_{11},\\
			\GW{\sigma_{011},[Z_{100}]}{\beta_1+\beta_3} = x_{12}, &
			\GW{\sigma_{011},[Z_{010}]}{\beta_1+\beta_3} = x_{13},	&
			\GW{\sigma_{011},[Z_{001}]}{\beta_1+\beta_3} = x_{14},\\
			\GW{\sigma_{110},[Z_{100}]}{2\beta_1} = y_{14},		&
			\GW{\sigma_{110},[Z_{010}]}{2\beta_1} = y_{15},			&
			\GW{\sigma_{110},[Z_{001}]}{2\beta_1} = y_{16},\\
			\GW{\sigma_{101},[Z_{100}]}{2\beta_1} = y_{17}, &
			\GW{\sigma_{101},[Z_{010}]}{2\beta_1} = y_{18}, &
			\GW{\sigma_{101},[Z_{001}]}{2\beta_1} = y_{19},\\
			\GW{\sigma_{011},[Z_{100}]}{2\beta_1} = y_{20}, &
			\GW{\sigma_{011},[Z_{010}]}{2\beta_1} = y_{21}, &
			\GW{\sigma_{011},[Z_{001}]}{2\beta_1} = y_{22},\\
			\GW{\sigma_{110},[pt]}{2\beta_1+\beta_3} = x_{24}, &
			\GW{\sigma_{101},[pt]}{2\beta_1+\beta_3} = x_{25}, &
			\GW{\sigma_{011},[pt]}{2\beta_1+\beta_3} = x_{26},\\
			\GW{\sigma_{110},[pt]}{\beta_1+2\beta_3} = x_{27}, &
			\GW{\sigma_{101},[pt]}{\beta_1+2\beta_3} = x_{28}, &
			\GW{\sigma_{011},[pt]}{\beta_1+2\beta_3} = x_{29},\\
			\GW{\sigma_{110},[pt]}{\beta_1+\beta_2} = y_{23}, &
			\GW{\sigma_{101},[pt]}{\beta_1+\beta_2} = y_{24}, &
			\GW{\sigma_{011},[pt]}{\beta_1+\beta_2} = y_{25},\\
			\GW{\sigma_{110},[pt]}{3\beta_1} = y_{26}, &
		\GW{\sigma_{101},[pt]}{3\beta_1} = y_{27}, &
		\GW{\sigma_{011},[pt]}{3\beta_1} = y_{28}\\ \end{array}$$ We can then write
		the quantum products: \begin{align} \sigma_{010}\ast\sigma_{110}	&= q_1
			y_5 \sigma_{110} + q_1 y_6 \sigma_{101} + q_1 y_7 \sigma_{011} + (q_1q_3
			x_6 + 2 q_1^2 y_{14})\sigma_{100}\\ \nonumber &+ (q_1q_3 x_7 + 2 q_1^2
			y_{15})\sigma_{010} + (q_1q_3 x_8 + 2 q_1^2 y_{16})\sigma_{001}\\
			\nonumber &+ 2 q_1^2q_3 x_{24} + q_1q_3^2 x_{27} + q_1q_2 y_{23} + 3
			q_1^3 y_{26}\\ \sigma_{010}\ast\sigma_{101}	&= [pt] + q_1 y_8
			\sigma_{110} + q_1 y_9 \sigma_{101} + q_1 y_{10} \sigma_{011} + (q_1q_3
			x_9 + 2 q_1^2 y_{17})\sigma_{100}\\ \nonumber &+ (q_1q_3 x_{10} + 2 q_1^2
			y_{18})\sigma_{010} + (q_1q_3 x_{11} + 2 q_1^2 y_{19})\sigma_{001}\\
			\nonumber				&+ 2 q_1^2q_3 x_{25} + q_1q_3^2 x_{28} + q_1q_2 y_{24} +
			3 q_1^3 y_{27}\\ \sigma_{010}\ast\sigma_{011}	&= [pt] + q_1 y_{11}
			\sigma_{110} + q_1 y_{12} \sigma_{101} + q_1 y_{13} \sigma_{011} +
			(q_1q_3 x_{12} + 2 q_1^2 y_{20})\sigma_{100}\\ \nonumber &+ (q_1q_3
		x_{13} + 2 q_1^2 y_{21})\sigma_{010} + (q_1q_3 x_{14} + 2 q_1^2
		y_{22})\sigma_{001}\\ \nonumber &+ 2 q_1^2q_3 x_{26} + q_1q_3^2 x_{29} +
		q_1q_2 y_{25} + 3 q_1^3 y_{28} \end{align}

\item For $\varepsilon = 111$, the curve classes which contribute with possibly
	non-zero Gromov-Witten invariants are \begin{align*}
		\beta_1,\beta_1+\beta_3,2\beta_1,2\beta_1+\beta_3,\beta_1+2\beta_3,\beta_1+\beta_2,3\beta_1\\
		3\beta_1+\beta_3,2\beta_1+2\beta_3,\beta_1+3\beta_3,2\beta_1+\beta_2,4\beta_1,\beta_1+\beta_2+\beta_3
	\end{align*} Using the fundamental class axiom,
	$\GW{\sigma_{010},[pt],[Z]}{\beta_1} = 0$. We record the remaining
	Gromov-Witten invariants.  $$\begin{array}{ccc}
		\GW{[pt],[Z_{110}]}{\beta_1+\beta_3} = x_{36},		&
		\GW{[pt],[Z_{101}]}{\beta_1+\beta_3} = x_{37},		&
		\GW{[pt],[Z_{011}]}{\beta_1+\beta_3} = x_{38},\\
		\GW{[pt],[Z_{110}]}{2\beta_1} = y_{29},			& \GW{[pt],[Z_{101}]}{2\beta_1}
		= y_{30}, & \GW{[pt],[Z_{011}]}{2\beta_1} = y_{31},\\
		\GW{[pt],[Z_{100}]}{2\beta_1+\beta_3} = x_{42},	&
		\GW{[pt],[Z_{010}]}{2\beta_1+\beta_3} = x_{43},		&
		\GW{[pt],[Z_{001}]}{2\beta_1+\beta_3} = x_{44},\\
		\GW{[pt],[Z_{100}]}{\beta_1+2\beta_3} = x_{45},	&
		\GW{[pt],[Z_{010}]}{\beta_1+2\beta_3} = x_{46},		&
		\GW{[pt],[Z_{001}]}{\beta_1+2\beta_3} = x_{47},\\
		\GW{[pt],[Z_{100}]}{\beta_1+\beta_2} = y_{32},		&
		\GW{[pt],[Z_{010}]}{\beta_1+\beta_2} = y_{33},		&
		\GW{[pt],[Z_{001}]}{\beta_1+\beta_2} = y_{34},\\
		\GW{[pt],[Z_{100}]}{3\beta_1} = y_{35},			& \GW{[pt],[Z_{010}]}{3\beta_1}
		= y_{36}, & \GW{[pt],[Z_{001}]}{3\beta_1} = y_{37},\\
		\GW{[pt],[pt]}{3\beta_1+\beta_3} = x_{54}, &
		\GW{[pt],[pt]}{2\beta_1+2\beta_3} = x_{55}, &
		\GW{[pt],[pt]}{\beta_1+3\beta_3} = x_{56},\\
		\GW{[pt],[pt]}{2\beta_1+\beta_2} = y_{38}, & \GW{[pt],[pt]}{4\beta_1} =
		y_{39}, & \GW{[pt],[pt]}{\beta_1+\beta_2+\beta_3} = x_{59} \end{array}$$ We
		record the final quantum product for the matrix $B$: \begin{align}
			\sigma_{010}\ast[pt]	&= (q_1q_3 x_{36} + 2 q_1^2 y_{29})\sigma_{110} +
			(q_1q_3 x_{37} + 2 q_1^2 y_{30})\sigma_{101} + (q_1q_3 x_{38} + 2 q_1^2
			y_{31})\sigma_{011}\\ \nonumber		&+ (2 q_1^2q_3 x_{42} + q_1q_3^2 x_{45}
			+ q_1q_2 y_{32} + 3 q_1^3 y_{35})\sigma_{100}\\ \nonumber		&+ (2
			q_1^2q_3 x_{43} + q_1q_3^2 x_{46} + q_1q_2 y_{33} + 3 q_1^3
			y_{36})\sigma_{010}\\ \nonumber &+ (2 q_1^2q_3 x_{44} + q_1q_3^2 x_{47} +
			q_1q_2 y_{34} + 3 q_1^3 y_{37})\sigma_{001}\\ \nonumber		&+ 3 q_1^3q_3
			x_{54} + 2 q_1^2q_3^2 x_{55} + q_1q_3^3 x_{56} + 2 q_1^2q_2 y_{38} + 4
		q_1^4 y_{39} + q_1q_2q_3 x_{59} \end{align} \end{itemize}

\subsection{Chevalley matrix $C$} As in the previous two sections, we can
express the quantum product as follows $$\sigma_{001}\ast\sigma_{\varepsilon} =
\sum_{\beta,\varepsilon'}
\GW{\sigma_{001},\sigma_{\varepsilon},[Z_{\varepsilon'}]}{\beta}\> q^{\beta}
\sigma_{\varepsilon'}$$

\begin{itemize} \item For $\varepsilon = 000$, since $1 = [Z]\in H^*(Z)$ is
			also the identity in $QH^*(Z)$, we have \begin{align} \sigma_{001}\ast 1
			= \sigma_{001} \end{align}

\item For $\varepsilon = 100,010,001$, we use the divisor axiom twice to reduce
	three-point Gromov-Witten invariants to one-point invariants. The applicable
		curve classes are $$\beta_3, \beta_1+\beta_3, \beta_2, 2\beta_3$$ In the
		first subsection, we observed that the one-point invariants for curve class
		$\beta_3$ are explicitly computable: \begin{align*} \GW{[Z_{100}]}{\beta_3}
			&= \int_{Z_{101}}[Z_{100}] = -1\\ \GW{[Z_{010}]}{\beta_3} &=
			\int_{Z_{101}}[Z_{010}] = 1\\ \GW{[Z_{001}]}{\beta_3} &=
		\int_{Z_{101}}[Z_{001}] = 0 \end{align*} In fact, all the relevant
		Gromov-Witten invariants (except $\GW{[pt]}{\beta_2}$) for these curve
		classes were already considered in the subsection for matrix $A$.  Using
		Corollary \ref{fibercor}, we compute $\GW{[pt]}{\beta_2} = 1$. We record
		the quantum products \begin{align} \sigma_{001}\ast\sigma_{100}	&=
			\sigma_{101} + q_3 \sigma_{100} - q_3 \sigma_{010} + (-q_1q_3 x_1 - 4
			q_3^2 x_2)\\ \sigma_{001}\ast\sigma_{010}	&= \sigma_{011} - q_1q_3 x_1\\
			\sigma_{001}\ast\sigma_{001}	&= \sigma_{011} - 2 \sigma_{101} - q_3
			\sigma_{100} + q_3 \sigma_{010} + (q_1q_3 x_1 + 4 q_3^2 x_2 + q_2)
		\end{align}

\item For $\varepsilon = 110,101,011$, we use the divisor axiom to reduce
	three-point Gromov-Witten invariants to two-point invariants. The applicable
		curve classes are very similar to those in matrix $A$, however
		$\beta_2+\beta_3 = [Z_{100}]$. So, by the divisor axiom
		$\GW{\sigma_{001},\sigma_{\varepsilon},[Z_{\varepsilon'}]}{\beta_2+\beta_3}
		= 0$. The applicable curve classes are
		$$\beta_3,\beta_1+\beta_3,\beta_2,2\beta_3,2\beta_1+\beta_3,\beta_1+2\beta_3,\beta_1+\beta_2,3\beta_3$$
		Recall, curve neighborhoods are used to show
		$\GW{\sigma_{\varepsilon},[Z_{\varepsilon'}]}{\beta_3} = 0$ unless
		$\varepsilon = 101$. We record the remaining Gromov-Witten invariants here
		$$\begin{array}{ccc} \GW{\sigma_{101},[Z_{110}]}{\beta_3} = x_3, &
			\GW{\sigma_{101},[Z_{101}]}{\beta_3} = x_4, &
			\GW{\sigma_{101},[Z_{011}]}{\beta_3} = x_5,\\
			\GW{\sigma_{110},[Z_{100}]}{\beta_1+\beta_3} = x_6, &
			\GW{\sigma_{110},[Z_{010}]}{\beta_1+\beta_3} = x_7, &
			\GW{\sigma_{110},[Z_{001}]}{\beta_1+\beta_3} = x_8,\\
			\GW{\sigma_{101},[Z_{100}]}{\beta_1+\beta_3} = x_9, &
			\GW{\sigma_{101},[Z_{010}]}{\beta_1+\beta_3} = x_{10},	&
			\GW{\sigma_{101},[Z_{001}]}{\beta_1+\beta_3} = x_{11},\\
			\GW{\sigma_{011},[Z_{100}]}{\beta_1+\beta_3} = x_{12}, &
			\GW{\sigma_{011},[Z_{010}]}{\beta_1+\beta_3} = x_{13},	&
			\GW{\sigma_{011},[Z_{001}]}{\beta_1+\beta_3} = x_{14},\\
			\GW{\sigma_{110},[Z_{100}]}{\beta_2} = z_1, &
			\GW{\sigma_{110},[Z_{010}]}{\beta_2} = z_2, &
			\GW{\sigma_{110},[Z_{001}]}{\beta_2} = z_3,\\
			\GW{\sigma_{101},[Z_{100}]}{\beta_2} = z_4, &
			\GW{\sigma_{101},[Z_{010}]}{\beta_2} = z_5, &
			\GW{\sigma_{101},[Z_{001}]}{\beta_2} = z_6,\\
			\GW{\sigma_{011},[Z_{100}]}{\beta_2} = z_7, &
			\GW{\sigma_{011},[Z_{010}]}{\beta_2} = z_8, &
			\GW{\sigma_{011},[Z_{001}]}{\beta_2} = z_9,\\
			\GW{\sigma_{110},[Z_{100}]}{2\beta_3} = x_{15}, &
			\GW{\sigma_{110},[Z_{010}]}{2\beta_3} = x_{16}, &
			\GW{\sigma_{110},[Z_{001}]}{2\beta_3} = x_{17},\\
			\GW{\sigma_{101},[Z_{100}]}{2\beta_3} = x_{18}, &
			\GW{\sigma_{101},[Z_{010}]}{2\beta_3} = x_{19}, &
			\GW{\sigma_{101},[Z_{001}]}{2\beta_3} = x_{20},\\
			\GW{\sigma_{011},[Z_{100}]}{2\beta_3} = x_{21}, &
			\GW{\sigma_{011},[Z_{010}]}{2\beta_3} = x_{22}, &
			\GW{\sigma_{011},[Z_{001}]}{2\beta_3} = x_{23},\\
			\GW{\sigma_{110},[pt]}{2\beta_1+\beta_3} = x_{24}, &
			\GW{\sigma_{101},[pt]}{2\beta_1+\beta_3} = x_{25}, &
			\GW{\sigma_{011},[pt]}{2\beta_1+\beta_3} = x_{26},\\
			\GW{\sigma_{110},[pt]}{\beta_1+2\beta_3} = x_{27}, &
			\GW{\sigma_{101},[pt]}{\beta_1+2\beta_3} = x_{28}, &
			\GW{\sigma_{011},[pt]}{\beta_1+2\beta_3} = x_{29},\\
			\GW{\sigma_{110},[pt]}{\beta_1+\beta_2} = y_{23}, &
			\GW{\sigma_{101},[pt]}{\beta_1+\beta_2} = y_{24}, &
		\GW{\sigma_{011},[pt]}{\beta_1+\beta_2} = y_{25},\\
		\GW{\sigma_{110},[pt]}{3\beta_3} = x_{33}, &
		\GW{\sigma_{101},[pt]}{3\beta_3} = x_{34}, &
		\GW{\sigma_{011},[pt]}{3\beta_3} = x_{35} \end{array}$$ And the quantum
		products are \begin{align} \sigma_{001}\ast\sigma_{110}	&= [pt] + (-q_1q_3
			x_6 + q_2 z_1 - 2 q_3^2 x_{15})\sigma_{100} + (-q_1q_3 x_7 + q_2 z_2 - 2
			q_3^2 x_{16})\sigma_{010}\\ \nonumber &+ (-q_1q_3 x_8 + q_2 z_3 - 2 q_3^2
			x_{17})\sigma_{001} + (- q_1^2q_3 x_{24} - 2 q_1q_3^2 x_{27} + q_1q_2
			y_{23} - 3 q_3^3 x_{33})\\ \sigma_{001}\ast\sigma_{101} &= [pt] - q_3 x_3
			\sigma_{110} - q_3 x_4 \sigma_{101} - q_3 x_5 \sigma_{011}\\ \nonumber &+
			(-q_1q_3 x_9 + q_2 z_4 - 2 q_3^2 x_{18})\sigma_{100} + (-q_1q_3 x_{10} +
			q_2 z_5 - 2 q_3^2 x_{19})\sigma_{010}\\ \nonumber &+ (-q_1q_3 x_{11} +
			q_2 z_6 - 2 q_3^2 x_{20})\sigma_{001} + (- q_1^2q_3 x_{25} - 2 q_1q_3^2
			x_{28} + q_1q_2 y_{24} - 3 q_3^3 x_{34})\\ \sigma_{001}\ast\sigma_{011}
			&= -[pt] + (-q_1q_3 x_{12} + q_2 z_7 - 2 q_3^2 x_{21})\sigma_{100} +
			(-q_1q_3 x_{13} + q_2 z_8 - 2 q_3^2 x_{22})\sigma_{010}\\ \nonumber &+
		(-q_1q_3 x_{14} + q_2 z_9 - 2 q_3^2 x_{23})\sigma_{001} + (- q_1^2q_3
		x_{26} - 2 q_1q_3^2 x_{29} + q_1q_2 y_{25} - 3 q_3^3 x_{35}) \end{align}

\item For $\varepsilon = 111$, the curve classes which contribute with possibly
	nonzero Gromov-Witten invariants are: \begin{align*}
		\beta_3,\beta_1+\beta_3,\beta_2,2\beta_3,2\beta_1+\beta_3,\beta_1+2\beta_3,\beta_1+\beta_2,3\beta_3\\
		3\beta_1+\beta_3, 2\beta_1+2\beta_3, \beta_1+3\beta_3, 2\beta_2,
	\beta_2+2\beta_3, 4\beta_3 \end{align*} Using the fundamental class axiom,
	$\GW{\sigma_{001},[pt],[Z]}{\beta_3} = 0$. We record the remaining unknown
	Gromov-Witten invariants.  $$\begin{array}{ccc}
		\GW{[pt],[Z_{110}]}{\beta_1+\beta_3} = x_{36},		&
		\GW{[pt],[Z_{101}]}{\beta_1+\beta_3} = x_{37},		&
		\GW{[pt],[Z_{011}]}{\beta_1+\beta_3} = x_{38},\\
		\GW{[pt],[Z_{110}]}{\beta_2} = z_{10}, & \GW{[pt],[Z_{101}]}{\beta_2} =
		z_{11}, & \GW{[pt],[Z_{011}]}{\beta_2} = z_{12},\\
		\GW{[pt],[Z_{110}]}{2\beta_3} = x_{39}, & \GW{[pt],[Z_{101}]}{2\beta_3} =
		x_{40}, & \GW{[pt],[Z_{011}]}{2\beta_3} = x_{41},\\
		\GW{[pt],[Z_{100}]}{2\beta_1+\beta_3} = x_{42},	&
		\GW{[pt],[Z_{010}]}{2\beta_1+\beta_3} = x_{43}, &
		\GW{[pt],[Z_{001}]}{2\beta_1+\beta_3} = x_{44},\\
		\GW{[pt],[Z_{100}]}{\beta_1+2\beta_3} = x_{45},	&
		\GW{[pt],[Z_{010}]}{\beta_1+2\beta_3} = x_{46}, &
		\GW{[pt],[Z_{001}]}{\beta_1+2\beta_3} = x_{47},\\
		\GW{[pt],[Z_{100}]}{\beta_1+\beta_2} = y_{32}, &
		\GW{[pt],[Z_{010}]}{\beta_1+\beta_2} = y_{33}, &
		\GW{[pt],[Z_{001}]}{\beta_1+\beta_2} = y_{34},\\
		\GW{[pt],[Z_{100}]}{3\beta_3} = x_{51}, & \GW{[pt],[Z_{010}]}{3\beta_3} =
		x_{52}, & \GW{[pt],[Z_{001}]}{3\beta_3} = x_{53},\\
		\GW{[pt],[pt]}{3\beta_1+\beta_3} = x_{54}, &
	\GW{[pt],[pt]}{2\beta_1+2\beta_3} = x_{55}, &
	\GW{[pt],[pt]}{\beta_1+3\beta_3} = x_{56},\\ \GW{[pt],[pt]}{2\beta_2} =
	z_{13}, & \GW{[pt],[pt]}{\beta_2+2\beta_3} = x_{57}, &
	\GW{[pt],[pt]}{4\beta_3} = x_{58} \end{array}$$ The final quantum product is
	\begin{align} \sigma_{001}\ast[pt]	&= (-q_1q_3 x_{36} + q_2 z_{10} - 2 q_3^2
		x_{39})\sigma_{110} + (-q_1q_3 x_{37} + q_2 z_{11} - 2 q_3^2
		x_{40})\sigma_{101}\\ \nonumber		&+ (-q_1q_3 x_{38} + q_2 z_{12} - 2 q_3^2
		x_{41})\sigma_{011} + (-q_1^2q_3 x_{42} - 2 q_1q_3^2 x_{45} + q_1q_2 y_{32}
		- 3 q_3^3 x_{51})\sigma_{100}\\ \nonumber		&+ (-q_1^2q_3 x_{43} - 2
		q_1q_3^2 x_{46} + q_1q_2 y_{33} - 3 q_3^3 x_{52})\sigma_{010}\\ \nonumber
		&+ (-q_1^2q_3 x_{44} - 2 q_1q_3^2 x_{47} + q_1q_2 y_{34} - 3 q_3^3
		x_{53})\sigma_{001}\\ \nonumber		&+ -q_1^3q_3 x_{54} - 2 q_1^2q_3^2 x_{55}
		- 3 q_1q_3^3 x_{56} + 2 q_2^2 z_{13} - 2 q_2q_3^2 x_{57} - 4 q_3^4 x_{58}
	\end{align} \end{itemize}

\subsection{Brute force} The unknown Gromov-Witten invariants can be computed
by imposing the relations $$[A,B] = 0,\quad [A,C] = 0,\quad [B,C] = 0;$$ these
are the relations that quantum multiplication commutes. This gives relations
among the remaining unknown invariants which can then be solved using brute
force. This gives values for all but a single Gromov-Witten invariant $$y_3 =
\GW{[Z_{001}]}{\beta_1}$$

We record the matrices here: ($A$ is the matrix obtained from multiplication by
$\sigma_{100}$, $B$ from multiplication by $\sigma_{010}$, and $C$ from
multiplication by $\sigma_{001}$) $$A = \begin{pmatrix} 0	& q_1q_3y_3	&
	q_1q_3y_3	&	-q_1q_3y_3	& 0			&	0 &	0 &	q_1q_2q_3y_3\\ 1	&	-q_3			& 0
	&	q_3			& q_1q_3y_3	&	q_1q_3y_3	&	0	&	0\\ 0 &	q_3			&	0			& -q_3 &	0
	&	0 &	0	&	0\\ 0	&	0			& 0			&	0			& q_1q_3y_3	& q_1q_3y_3	&	0	&	0\\ 0 &
	0			&	1			& 0			&	0			&	q_3 &	0 &	0\\ 0	&	0			& 0			&	1			&	0 &
-q_3			&	0	&	0\\ 0	& 0			&	0 &	0 &	0			&	q_3			&	0 &	q_1q_3y_3\\ 0	&	0
&	0 &	0			&	0			& 0 &	1	&	0 \end{pmatrix}$$ $$B = \begin{pmatrix} 0	&
	q_1q_3y_3	& q_1q_3y_3	&	-q_1q_3y_3	& 0			&	0			& q_1q_2y_3	&
	q_1q_2q_3y_3\\ 0	&	0 &	2q_1y_3		&	0			& q_1q_3y_3	& q_1q_3y_3	&	0 &
	q_1q_2y_3\\ 1	&	0			&	-q_1y_3 &	0			&	0 & 0			&	0			&	0\\ 0 &	0			&
	q_1y_3		&	0 &	q_1q_3y_3	& q_1q_3y_3	&	0 &	0\\ 0	&	1			&	1 &	0			&
	-q_1y_3		&	0 &	0 &	0\\ 0	&	0 &	0			&	0			& q_1y_3		&	0			&	0 &	0\\ 0
&	0 &	0 &	1			&	0			& 0			&	0			& q_1q_3y_3\\ 0	&	0			&	0 &	0			&	0
&	1 &	1			&	0 \end{pmatrix}$$ $$C = \begin{pmatrix} 0	&	-q_1q_3y_3	&
	-q_1q_3y_3	& q_1q_3y_3+q_2	& 0			&	0 &	q_1q_2y_3	&	0\\ 0	&	q_3 &	0			&
	-q_3 & -q_1q_3y_3	&	-q_1q_3y_3+q_2	&	0 &	q_1q_2y_3\\ 0	& -q_3			&	0 &	q_3
	&	0 &	0				&	q_2 &	0\\ 1	& 0			&	0 &	0				&	-q_1q_3y_3	& -q_1q_3y_3
	&	0 &	0\\ 0 &	0			&	0			& 0				&	0			& -q_3 &	0			& q_2\\ 0	&	1
	&	0 &	-2				&	0 &	q_3 &	0 &	0\\ 0	&	0			&	1 &	1				&	0 &	-q_3
	&	0 &	-q_1q_3y_3\\ 0	&	0			&	0 &	0				&	1 &	1				&	-1 &	0
\end{pmatrix}$$

In order to compute $y_3$, we use \cite[Remark 5.7]{Manolache}. In particular,
this remark immediately implies the following result.

\begin{prop} Let $Z$ denote the Bott-Samelson variety
	$Z(\alpha_1,\alpha_2,\alpha_1)$, and $Z' = Z(\alpha_1,\alpha_2)$.  The
	following commutative diagram is Cartesian: $$\begin{tikzcd}
		\moduli{1}{Z}{\beta_1} \arrow{r}{\overline{\theta}}
		\arrow{d}{\overline{\pi}}	& \moduli{1}{G/B}{[X(s_{\alpha_2})]}
	\arrow{d}{\overline{p_{\alpha_1}}}\\ \moduli{1}{Z'}{[Z'_{01}]}
	\arrow{r}{\overline{p_{\alpha_1}\theta'}}			&
	\moduli{1}{G/P_{\alpha_1}}{[X(s_{\alpha_2})]} \end{tikzcd}$$ In particular,
	since $\overline{p_{\alpha_1}}$ is an isomorphism,
$\overline{\pi}:\moduli{1}{Z}{\beta_1}\to \moduli{1}{Z'}{[Z'_{01}]}$ is also an
isomorphism.  \end{prop}

Combined with Corollary \ref{fibercor}, we can compute $y_3$: \begin{align*}
y_3 &= \GW{[Z_{001}]}{\beta_1} = \int ev^*(\sigma_{110})\cdot
[\moduli{1}{Z}{\beta_1}]\\ &= \int ev^*([pt])\cdot [\moduli{1}{Z'}{[Z'_{01}]} =
\int_{Z'}[pt] = 1.  \end{align*}

\begin{prop} The Gromov-Witten invariant $y_3 = \GW{[Z_{001}]}{\beta_1} = 1$.
\end{prop}

One can read the entire Chevalley formula for $Z=Z(\alpha_1,\alpha_2,\alpha_1)$
from the Chevalley matrices now that $y_3$ has been computed; we record the
presentation for $QH^*(Z)$ so obtained here:

\begin{thm} $QH^*(Z)$ is generated by
	$\sigma_{100},\sigma_{010},\sigma_{001},q_1,q_2,q_3$ subject to the following
	relations: \begin{align*} \sigma_{100}^2 &= q_1q_3 - q_3\sigma_{100} +
		q_3\sigma_{010}\\ \sigma_{010}^2 &= q_1q_3 + 2q_1\sigma_{100} -
		q_1\sigma_{010} + q_1\sigma_{001} + \sigma_{110}\\ \sigma_{001}^2 &= q_1q_3
	+ q_2 - q_3\sigma_{100} + q_3\sigma_{010} - 2\sigma_{101} + \sigma_{011}
	\end{align*} \end{thm}

We can also record the ``Giambelli formula,'' the representation of the vector
	space generators $\sigma_{000},\sigma_{100},\ldots,\sigma_{011},\sigma_{111}$
	as polynomials in the algebra generators
	$\sigma_{100},\sigma_{010},\sigma_{001}$.

\begin{cor} In $QH^*(Z)$, the Giambelli formulae for the classes
	$$\sigma_{110},\sigma_{101},\sigma_{011},\sigma_{111}$$ are as follows:
	\begin{align*} \sigma_{110} &= \sigma_{100}\sigma_{010} - q_1q_3\\
		\sigma_{101} &= \sigma_{100}\sigma_{001} - q_3\sigma_{100} +
	q_3\sigma_{010} + q_1q_3\\ \sigma_{011} &= \sigma_{010}\sigma_{001} +
	q_1q_3\\ \sigma_{111} &= \sigma_{100}\sigma_{010}\sigma_{001} +
	q_1q_3\sigma_{100} \end{align*} \end{cor}

Using these formulas, we are able to verify that the ring $QH^*(Z)$ is indeed
	associative.

\section{Conjecture $\mathcal{O}$}
	Consider the operator $\hat{c}_1:H^*(Z)\to H^*(Z)$ defined as follows:
	let $c_1:QH^*(Z)\to QH^*(Z)$ be the operator defined by multiplication
	by $c_1(-K_Z)$, and let $\hat{c}_1$ denote the specialization of $c_1$
	with all quantum parameters set equal to one.

	Conjecture $\mathcal{O}$, which is related to the Gamma conjectures of
	Galkin, Golyshev, and Iritani (see \cite{conjOhomog}), is a statement
	concerning the eigenvalues of $\hat{c}_1$. Namely, Conjecture
	$\mathcal{O}$ states that if $Z$ is Fano, then the following properties
	hold:
	\begin{enumerate}
		\item Let $\delta_0$ denote the maximum modulus among the
			eigenvalues of $\hat{c}_1$. Then $\delta_0$ is one of
			the eigenvalues of $\hat{c}_1$, and occurs with
			multiplicity one;
		\item If $\delta$ is any eigenvalue of $\hat{c}_1$ with
			$|\delta| = \delta_0$, then there is an $r$th root of
			unity $\zeta$ such that $\delta = \delta_0\zeta$, where
			$r$ is the Fano index of $Z$.
	\end{enumerate}

	Since the Bott-Samelson variety $Z = Z(\alpha_1,\alpha_2,\alpha_1)$ is
	Fano (see Remark \ref{Fano}), and using the Chevalley matrices from the
	last section, we can write the matrix for $\hat{c}_1$:
	$$ \hat{c}_1 = \begin{pmatrix}
		0 & 2 & 2 & -2 & 0 & 0 & 3 & 4 \\
		3 & -1 & 2 & 1 & 2 & 4 & 0 & 3 \\
		1 & 1 & -1 & -1 & 0 & 0 & 2 & 0 \\
		2 & 0 & 1 & 0 & 2 & 2 & 0 & 0 \\
		0 & 1 & 4 & 0 & -1 & 1 & 0 & 2 \\
		0 & 2 & 0 & -1 & 1 & -1 & 0 & 0 \\
		0 & 0 & 2 & 3 & 0 & 1 & 0 & 2 \\
		0 & 0 & 0 & 0 & 2 & 3 & 2 & 0 \\
	\end{pmatrix} $$

	This matrix has eight distinct eigenvalues which are approximated
	numerically, the eigenvalue with the largest modulus is real, and all
	other eigenvalues have strictly smaller modulus. In particular,
	Conjecture $\mathcal{O}$ holds for the Bott-Samelson variety
	$Z(\alpha_1,\alpha_2,\alpha_1)$.

	\begin{thm}
		The Conjecture $\mathcal{O}$ holds for the Bott-Samelson
		variety $Z = Z(\alpha_1,\alpha_2,\alpha_1)$ (in Type $A_2$).
	\end{thm}

\bibliographystyle{amsalpha} \bibliography{references}

\end{document}